\documentclass[11pt,reqno]{amsart}
\usepackage{amsmath,amscd,amssymb,amsfonts,amsthm,courier,relsize,bm}
\usepackage{hyperref,enumerate,mathrsfs,mathtools,slashed}
\usepackage{scalerel}
\usepackage[all]{xy}

\usepackage{xcolor}  
\hypersetup{
    colorlinks,
    linkcolor={red!50!black},
    citecolor={blue!70!black},
    urlcolor={blue!80!black}
}

\newtheorem{theorem}{Theorem}[section]
\newtheorem{corollary}[theorem]{Corollary}
\newtheorem{proposition}[theorem]{Proposition}

\newtheorem{lemma}[theorem]{Lemma}
\theoremstyle{definition}    
\newtheorem{definition}[theorem]{Definition}
\theoremstyle{remark}

\newtheorem{remark}[theorem]{Remark}

\newcommand{\pair}[2]{\langle #1, #2 \rangle}
\newcommand{\ignore}[1]{}
\newcommand{\matr}[4]{\left(\begin{array}{cc}#1&#2\\#3&#4\end{array}\right)}
\newcommand{\ol}[1]{\overline{#1}}
\newcommand{\ul}[1]{\underline{#1}}
\newcommand{\sul}[1]{\underline{#1\mkern-4mu}\mkern4mu } 

\renewcommand{\sf}[1]{\mathsf{#1}}
\newcommand{\wh}[1]{\widehat{#1}}
\newcommand{\scr}[1]{\mathscr{#1}}
\newcommand{\mf}[1]{\mathfrak{#1}}

\newcommand{\tn}[1]{\textnormal{#1}}
\renewcommand{\i}{{\mathrm{i}}}
\renewcommand{\c}{{\mathrm{c}}}

\def\Dirac{\ensuremath{\slashed{\D}}}

\def\DiracS{\ensuremath{\Dirac{}^\Sigma}}
\def\DiracSP{\ensuremath{\Dirac{}^{\Sigma,+}}}
\def\DiracSM{\ensuremath{\Dirac{}^{\Sigma,-}}}
\def\DiracC{\ensuremath{\Dirac{}^C}}
\def\d{\ensuremath{\mathrm{d}}}

\def\Ad{\ensuremath{\textnormal{Ad}}}
\def\ad{\ensuremath{\textnormal{ad}}}
\def\g{\ensuremath{\mathfrak{g}}}

\def\t{\ensuremath{\mathfrak{t}}}
\def\n{\ensuremath{\mathfrak{n}}}

\def\hvee{\ensuremath{\textnormal{h}^\vee}}

\def\A{\ensuremath{\mathcal{A}}}
\def\B{\ensuremath{\mathcal{B}}}
\def\C{\ensuremath{\mathcal{C}}}
\def\D{\ensuremath{\mathcal{D}}}
\def\E{\ensuremath{\mathcal{E}}}

\def\G{\ensuremath{\mathcal{G}}}

\def\J{\ensuremath{\mathcal{J}}}

\def\M{\ensuremath{\mathcal{M}}}

\def\R{\ensuremath{\mathcal{R}}}
\def\S{\ensuremath{\mathcal{S}}}

\def\U{\ensuremath{\mathcal{U}}}
\def\V{\ensuremath{\mathcal{V}}}

\def\bC{\ensuremath{\mathbb{C}}}
\def\bR{\ensuremath{\mathbb{R}}}
\def\bZ{\ensuremath{\mathbb{Z}}}

\def\bA{\ensuremath{\mathbb{A}}}
\def\bD{\ensuremath{\mathbb{D}}}
\def\bE{\ensuremath{\mathbb{E}}}
\def\bQ{\ensuremath{\mathbb{Q}}}

\def\End{\ensuremath{\textnormal{End}}}
\def\Hom{\ensuremath{\textnormal{Hom}}}

\def\ker{\ensuremath{\textnormal{ker}}}

\def\id{\ensuremath{\textnormal{id}}}

\def\Tr{\ensuremath{\textnormal{Tr}}}

\def\dim{\ensuremath{\textnormal{dim}}}

\def\dom{\ensuremath{\textnormal{dom}}}

\def\pt{\ensuremath{\textnormal{pt}}}

\def\Cl{\ensuremath{\textnormal{Cl}}}

\def\index{\ensuremath{\textnormal{index}}}

\def\sK{\ensuremath{\textnormal{K}}}
\def\aff{\ensuremath{\textnormal{aff}}}

\def\Ch{\ensuremath{\textnormal{Ch}}}

\def\Ahat{\ensuremath{\widehat{\textnormal{A}}}}

\def\hol{\ensuremath{\tn{hol}}}
\def\hotimes{\ensuremath{\wh{\otimes}}}

\def\ft{\ensuremath{[\![t]\!]}}
\def\reg{\ensuremath{\tn{reg}}}

\usepackage[outline]{contour}
\newcommand*{\fancy}[1]{{\color{white}\contour{black}{#1}}}
\def\boldSigma{\ensuremath{\fancy{$\Sigma$}}}
\newcommand{\bSigma}{\mathchoice
  {\mbox{\scaleobj{0.95}{\boldSigma}}}
  {\mbox{\scaleobj{0.95}{\boldSigma}}}
  {\mbox{\scaleobj{0.7}{\boldSigma}}}
  {\mbox{\scaleobj{0.5}{\boldSigma}}}
}

\def\DiracbS{\ensuremath{\Dirac{}^{\bSigma}}}
\def\DiracbSP{\ensuremath{\Dirac{}^{\bSigma,+}}}

\def\sg{\ensuremath{\mathrm{g}}}

\def\bM{\ensuremath{\mathbb{M}}}

\title[Index theory of the Atiyah-Bott classes]{Families of elliptic boundary problems and index theory of the Atiyah-Bott classes}
\author{Yiannis Loizides}
\address{Department of Mathematical Sciences, George Mason University}
\email{yloizide@gmu.edu}
\begin{document}
\sloppy

\begin{abstract}
We study a natural family of non-local elliptic boundary problems on a compact oriented surface $\Sigma$ parametrized by the moduli space $\M_\Sigma$ of flat $G$-connections with framing along $\partial \Sigma$. This family generalizes one introduced by Atiyah and Bott for closed surfaces. In earlier work we constructed an analytic index morphism out of a subring of the K-theory of $\M_\Sigma$. In this article we apply that morphism to the K-class of the Fredholm family and derive cohomological formulas. The main application is to calculate K-theory intersection pairings on symplectic quotients of $\M_\Sigma$; the latter are compact moduli spaces of flat connections on surfaces with boundary, where the boundary holonomies lie in prescribed conjugacy classes. The results provide a gauge theory analogue of the Teleman-Woodward index formula. 
\end{abstract}

\maketitle

\section{Introduction}
Let $G$ be a compact connected simply connected simple Lie group with Lie algebra $\g$, and fix an invariant inner product `$\cdot$' on $\g$. Let $\bSigma$ be a closed connected Riemann surface of genus $\rm{g}$. Atiyah and Bott \cite{AtiyahBottYangMills} observed that the infinite dimensional affine space of connections $\A_{\bSigma}$ on the trivial bundle $P=\bSigma\times G$ has a canonical symplectic form
\[ \omega_{\A_{\bSigma}}(a,b)=\int_\Sigma a\cdot b, \quad a,b \in T_A \A_{\bSigma}\simeq \Omega^1(\bSigma,\g).\]
The action of the gauge group $\G_{\bSigma}=\tn{Map}(\bSigma,G)$ is Hamiltonian with moment map $\mu(A)=\tn{curv}(A)$, the curvature of the connection $A$. The zero fiber $\A_{\bSigma}^\flat=\tn{curv}^{-1}(0)$ is the space of flat connections on $\bSigma$, and thus the symplectic quotient $\bM_{\bSigma}=\A_{\bSigma}^\flat/\G_{\bSigma}$, is the finite dimensional moduli space of flat connections on $\bSigma$. Atiyah and Bott studied the cohomology of $\bM_{\bSigma}$ using the Morse theory of the Yang-Mills functional on $\A_{\bSigma}$.

Atiyah and Bott introduced classes in the cohomology of $\bM_{\bSigma}$, obtained by applying the Chern character to certain canonical K-theory classes on $\bM_{\bSigma}$. These Atiyah-Bott K-theory classes $\sf{E}^NV$ are constructed from a pair $(V,N)$ consisting of a finite dimensional complex representation $V$ of $G$ and a closed submanifold $N \subset \Sigma$. Consider the $\G_{\bSigma}$-equivariant vector bundle $\V=\A_{\bSigma}\times V \times \bSigma \rightarrow \A_{\bSigma}\times \bSigma$. Then $\sf{E}^NV$ is obtained by restricting $\V$ to $\A_{\bSigma}^\flat\times N$, dividing by the diagonal $\G_{\bSigma}$ action, and then applying the wrong-way map $p_!$ in K-theory associated to the projection $p\colon \bM_{\bSigma}\times N \rightarrow \bM_{\bSigma}$ and a choice of square root of the canonical bundle of $\bSigma$.

By a well-known theorem of Narasimhan and Seshadri, $\bM_{\bSigma}$ can be identified with the moduli space of semistable holomorphic $G_\bC$-bundles on $\bSigma$. Going beyond semistable bundles, one can study the moduli stack $\mf{M}_{\bSigma}$ of all holomorphic $G_\bC$-bundles. There are analogues of the Atiyah-Bott classes here, studied by Teleman and Woodward in \cite{TelemanWoodward}, building on earlier work of Teleman. These authors defined an index morphism on the ring of even degree K-theory classes generated by the Atiyah-Bott classes and proved a remarkable formula for the index, giving a vast generalization of the Verlinde formulas. One recent application of this result was to the proof of the equivariant Verlinde formula for the moduli space of Higgs bundles \cite{EquivVerlinde, EquivVerlindePf1, EquivVerlindePf2}.

In this article we will be interested in the case of a Riemann surface $\Sigma$ with non-empty boundary. Besides their intrinsic interest, moduli spaces of flat connections on Riemann surfaces with boundary are natural spaces to consider when studying behavior of invariants under gluing of surfaces, for example in topological quantum field theory. By restricting the holonomy of the connection around each boundary component to lie in specified conjugacy classes of $G$, one recovers moduli spaces of holomorphic $G_\bC$-bundles with (quasi)parabolic structures at a finite set of points, by a generalization of the Narasimhan-Seshadri correspondence due to Seshadri and Mehta.

Let $\Sigma$ be a connected Riemann surface with $b\ge 1$ boundary components. For $N$ a closed submanifold of $\Sigma$ of dimension $0$ or $1$, the definition of the Atiyah-Bott class $\sf{E}^NV$ is essentially unchanged. However when $N=\Sigma$ the wrong-way map $p_!$ is not defined. From the analytic point of view that we take in this article, the issue is that the $\ol{\partial}$ operator on a Riemann surface with boundary is not a Fredholm operator. This is apparent already for the disk $\bD \subset \bC$, when $\ker(\ol{\partial})$ is infinite dimensional, consisting of all holomorphic functions on the disk. As is well-known, to obtain a Fredholm operator one can impose a boundary condition. For the disk $\bD$, the boundary condition we impose is that the restriction of a $V$-valued function to $\partial \bD$ has vanishing non-negative Fourier modes. In fact the same boundary condition (imposed along each component of the boundary) works on a general $\Sigma$. This is a simple instance of the Atiyah-Patodi-Singer boundary condition \cite{AtiyahPatodiSingerI}.

Unfortunately there is no continuous gauge-equivariant choice of boundary condition of the general type described above for the family $\bar{\partial}_A$, $A \in \A_\Sigma$. We are thus forced to consider the moduli space $\M_\Sigma$ of flat connections with \emph{framing} along the boundary $\partial \Sigma$. The construction of $\M_\Sigma$ is similar to the construction of $\bM_{\bSigma}$ except that one only divides by the normal subgroup $\G_{\Sigma,\partial \Sigma} \subset\G_\Sigma$ consisting of gauge transformations that equal the identity along the boundary: 
\[ \M_\Sigma=\A_\Sigma^\flat/\G_{\Sigma,\partial \Sigma}. \] 
The space $\M_\Sigma$ is a smooth infinite dimensional symplectic manifold, with a residual action of the quotient $\G_\Sigma/\G_{\Sigma,\partial \Sigma}$. Choosing boundary parametrizations, restriction to the boundary gives a smooth map to the loop group
\begin{equation}
\label{e:resbd}
\G_\Sigma/\G_{\Sigma,\partial \Sigma}\rightarrow L\sul{G}, \qquad \sul{G}=G^b.
\end{equation}  
When $G$ is simply connected, \eqref{e:resbd} is surjective and hence an isomorphism. The action of $L\sul{G}$ on $\M_\Sigma$ is Hamiltonian, with moment map
\[ \mu_{\M_\Sigma}\colon \M_\Sigma \rightarrow \Omega^1(\partial \Sigma,\g)\simeq \Omega^1(S^1,\ul{\g}) \]
induced by pullback of connections to $\partial \Sigma$. The framed moduli space $\M_\Sigma$ appeared in work of Donaldson \cite{donaldson1992boundary}, and was studied systematically in work of Meinrenken and Woodward \cite{MWVerlindeFactorization} who developed a theory of Hamiltonian loop group spaces encompassing $\M_\Sigma$ and various other examples. For $\Sigma=\bD$, $\M_\Sigma$ is a smooth version of the affine Grassmannian for $G$, cf. \cite{PressleySegal}. Symplectic quotients of $\M_\Sigma$ are finite-dimensional, compact, but possibly singular, moduli spaces of flat connections on $\Sigma$ with boundary holonomies constrained to lie in prescribed conjugacy classes. These moduli spaces have a rich topology and are of great interest in symplectic and algebraic geometry and in physics. In particular when all conjugacy classes are chosen to be trivial, the symplectic quotient recovers $\bM_{\bSigma}$, where $\bSigma$ is the closed surface obtained by capping off the boundary.

The action of the subgroup $\G_{\Sigma,\partial \Sigma}$ preserves the boundary condition described above. The family of Fredholm operators, $\ol{\partial}_A$ with fixed boundary condition and varying $A \in \A_\Sigma^\flat\subset \A_\Sigma$, thus descends to a continuous $\sul{G}$-equivariant Fredholm family over $\M_\Sigma$. The latter family represents a $\sul{G}$-equivariant K-theory class $\sf{E}^\Sigma V \in K_{\sul{G}}^0(\M_\Sigma)$. For closed manifolds such `tautological' families of Dirac operators parametrized by moduli spaces of connections (or vector potentials) were considered for example by Atiyah and Singer \cite{atiyah1984dirac}, as well as many other authors in both the mathematics and physics literature. The corresponding families for manifolds with boundary have received less attention, and in particular we are not aware of any prior work studying the above family. The general index theory of families of Atiyah-Patodi-Singer boundary problems has been developed however, including the well-known index theorems of Bismut-Cheeger \cite{BismutCheegerFamilies} and Melrose-Piazza \cite{melrose1997families}.

As mentioned above, $\M_\Sigma$ is an instance of a proper Hamiltonian loop group space in the sense of Meinrenken and Woodward \cite{MWVerlindeFactorization}. In earlier work \cite{LIndexFormula}, we constructed an `analytic index homomorphism' in K-theory for any proper Hamiltonian loop group space. Specializing to $\M_\Sigma$, and choosing a `level' $k\ge 0$, this is a homomorphism
\begin{equation} 
\label{e:indexintro}
\index_{\sul{T}}^k \colon K^{\ad}_{\sul{T}}(\M_\Sigma)\rightarrow R^{-\infty}(\sul{T}).
\end{equation}
Here $T\subset G$ is a maximal torus, $\sul{T}=T^b\subset G^b=\sul{G}$, and $R^{-\infty}(\sul{T})$ is the space of formal Fourier series
\begin{equation}
\label{e:formalseriesT}
\sum_{\ul{\lambda}\in \ul{\Lambda}} n_{\ul{\lambda}} e^{\ul{\lambda}}, \qquad n_{\ul{\lambda}}\in \bZ, \quad \ul{\Lambda}=\Hom(\sul{T},U(1)).
\end{equation}
The homomorphism \eqref{e:indexintro} is defined by coupling the K-theory class to a first-order elliptic operator on a finite-dimensional submanifold $X \subset \M_\Sigma$, and taking the equivariant $L^2$-index. Informally one can think of this as pushing forward K-classes along the map $X\rightarrow \pt$ using a canonical K-orientation that exists on $X$. The level $k$ is an additional freedom to twist by the $k$-th tensor power of the canonical prequantum line bundle $L$ on $\M_\Sigma$. On K-classes in the image of the restriction map from $\sul{G}$ to $\sul{T}$-equivariant K-theory, the index takes values in $\sul{W}$-antisymmetric formal series, i.e. series \eqref{e:formalseriesT} that change by a sign $(-1)^{l(\ul{w})}$ under the action of the Weyl group $\sul{W}$ on $\ul{\Lambda}$. Applying induction thus gives a similar map
\begin{equation}
\index_{\sul{G}}^k\colon K^\ad_{\sul{G}}(\M_\Sigma)\rightarrow R^{-\infty}(\sul{G}),
\end{equation}
where $R^{-\infty}(\sul{G})$ is the $R(\sul{G})$-module of formal series of irreducible $\sul{G}$-characters.

The manifold $X$ is a variant of the `extended moduli spaces' used by Jeffrey and Kirwan \cite{JeffreyKirwanModuli} in their work on the cohomology ring of the moduli space of flat connections, and a related space was used by Meinrenken in \cite{Meinrenken2005}. The manner in which $X$ sits inside $\M_\Sigma$ is similar to the way $\t=\tn{Lie}(T)$ sits inside $\Omega^1(S^1,\g)$ as constant connections $\xi \d \theta$, $\xi \in \t$ on the circle. Like $\t$ itself, $X$ is \emph{non-compact}, and this is a fundamental complication in the story. One consequence is that \eqref{e:indexintro} is only defined on a subring $K^{\ad}_{\sul{T}}(\M_\Sigma)\subset K_{\sul{T}}(\M_\Sigma)$ of the even topological K-theory of $\M_\Sigma$ consisting of K-classes represented by Fredholm complexes satisfying an equivariant bounded geometry condition. For these admissible K-classes, the $L^2$ kernel and cokernel of the resulting operator on $X$ are infinite dimensional but with finite $\sul{T}$-multiplicities.

The index homomorphism is a tool for computing K-theory intersection pairings on symplectic quotients of $\M_\Sigma$. For symplectic geometers, this is the K-theory analogue of the more familiar fact that twisted Duistermaat-Heckman distributions encode cohomological intersection pairings on symplectic quotients. To explain the relationship more precisely, let $\sf{E} \in K^{\ad}_{\sul{G}}(\M_\Sigma)$. Let $\mf{A}\subset \t_+$ be a fundamental alcove. Using the basic inner product on $\g$, we identify $\t\simeq \t^*$, and hence the real weight lattice $\Lambda\simeq \Hom(T,U(1))$ is identified with a lattice in $\t$ containing the sub-lattice $\Pi=\ker(\exp_T)\simeq \Hom(U(1),T)$. Let $\ul{a}=(a_1,...,a_b) \in \ul{\mf{A}}=\mf{A}^{b}$ be a regular value of $\mu_{\M_\Sigma}$. Then the symplectic quotient 
\[ \M_{\Sigma,\ul{a}}=\mu_{\M_\Sigma}^{-1}(\ul{a})/(L\sul{G})_{\ul{a}} \] 
is a finite dimensional compact symplectic orbifold. Assume $\ul{a}\in \ul{\Lambda} \otimes \bQ$ and $k \in \bZ$ is such that $k\ul{a}\in \ul{\Lambda}$. If, in addition, the stabilizer  $(L\sul{G})_{\ul{a}}\subset \sul{G}$ (this occurs if and only if each $a_j$ is not in the far wall of the alcove $\mf{A}$), then $\sf{E}\otimes L^k\otimes \bC_{-k\ul{a}}$ is $(L\sul{G})_{\ul{a}}$-equivariant and hence descends to a K-class
\[ \sf{E}_{k\ul{a}} \in K^0(\M_{\Sigma,\ul{a}}). \]
Using any compatible almost complex structure on $\M_{\Sigma,\ul{a}}$, the index
\[ \index(\sf{E}_{k\ul{a}}) \in \bZ \]
is defined. Let $V_{k\ul{a}}$ be the irreducible $\sul{G}$ representation with highest weight $k\ul{a}$. The non-abelian localization theorem proved in \cite{LIndexFormula} implies (see also \cite{WittenNonAbelian}),
\begin{equation} 
\label{e:indexreduced}
\index(\sf{E}_{k\ul{a}})=\tn{multiplicity}\big(V_{k\ul{a}},\, \index^k_{\sul{G}}(\sf{E})\big), \quad \text{ for } k\gg 0.
\end{equation}
On the other hand, by the index formula for orbifolds, the left hand side of \eqref{e:indexreduced} is a quasi-polynomial function of $k$, hence is completely determined by its values for $k\gg 0$. In this sense the index homomorphism \eqref{e:indexintro} determines K-theory intersection pairings on $\M_{\Sigma,\ul{a}}$. In case $\sf{E}=1$, the $[Q,R]=0$ theorem for Hamiltonian loop group spaces implies that \eqref{e:indexreduced} holds already for $k\ge 1$; a result closely related to the Verlinde formulas, cf. \cite{AMWVerlinde} for details.

Further intuition for the index homomorphism comes from earlier work \cite{LoizidesGeomKHom, LMSspinor, LSQuantLG, LSWittenDef} where we studied the special case $\sf{E}=1$. In that case we proved that $\index^k_{\sul{T}}(1)$ is the Weyl-Kac numerator of a projective positive energy representation of $L\sul{G}$ that deserves to be thought of as the geometric quantization of the infinite dimensional symplectic manifold $\M_\Sigma$ (there are a few approaches to defining the latter, cf. \cite{AMWVerlinde, MeinrenkenKHomology, LSQuantLG}, all of which agree).

The contents of the article are as follows. Section 2 recalls the construction of the infinite dimensional symplectic manifold $\M_\Sigma$ associated to a compact connected oriented surface $\Sigma$ with genus $\rm{g}$ and $b\ge 1$ boundary components. Section 3 gives a careful discussion of the Atiyah-Bott classes $\sf{E}^NV$, with particular focus on the most non-trivial case $N=\Sigma$. Section 4 describes the index homomorphism \eqref{e:indexintro} and establishes that the Fredholm complexes representing the Atiyah-Bott classes satisfy the admissibility hypothesis. The main results on $\sf{E}^\Sigma V$ from Sections 3 and 4 are summarized in the following theorem.
\begin{theorem}
The family of operators $\ol{\partial}_A$, $A \in \A_\Sigma^\flat$ with the boundary condition described above, descends to a continuous $\sul{G}$-equivariant family of Fredholm operators on $\M_\Sigma=\A_\Sigma^\flat/\G_{\Sigma,\partial \Sigma}$ representing an even admissible K-class $\sf{E}^\Sigma V \in K_{\sul{G}}^\ad(\M_\Sigma)$. The image of $\sf{E}^\Sigma V$ in $K_{\sul{T}}^\ad(\M_\Sigma)$ transforms under the action of an element $(\ul{w},\ul{\eta})\in \sul{W}\ltimes \ul{\Pi}=\sul{W}_{\tn{af}}$ of the affine Weyl group of $\sul{G}$ as
\begin{equation} 
\label{e:indexchangeaffweyl}
(\ul{w},\ul{\eta})^*\sf{E}^\Sigma V=\sf{E}^\Sigma V-(2\pi \i)^{-1}\partial_{\ul{\eta}}\ul{\Tr}_V,
\end{equation}
where $\Tr_V$ is the character of $V$, $\ul{\Tr}_V(\sul{g})=\Sigma_i \Tr_V(g_i)$ for $\sul{g}=(g_1,...,g_b)\in \sul{G}$, and $\partial_{\ul{\eta}}\ul{\Tr}_V$ is the derivative.
\end{theorem}
Note that $(2\pi \i)^{-1}\partial_{\ul{\eta}}\ul{\Tr}_V\in R(\sul{T})\simeq R(\sul{T})\cdot 1 \subset K^0_{\sul{T}}(\M_\Sigma)$. The non-trivial transformation \eqref{e:indexchangeaffweyl} is calculated using a general formula for the change in the Fredholm index of an elliptic boundary problem under a change of boundary conditions.

In Section 5 we use the families index theorem to construct an equivariant differential form representing the equivariant Chern character of $\sf{E}^\Sigma V\upharpoonright X$. The resulting differential form possesses good analytic properties, and can be (equivariantly) integrated over the non-compact manifold $X$, leading to a Kirillov-type formula for the index. Following Teleman and Woodward, Section 6 studies the generating series
\[ \sf{E}_t=\exp(t\sf{E}^\Sigma V) \in \bC K^\ad_{\sul{G}}(\M_\Sigma)\ft \]
where $t$ is a formal variable. The exponential exchanges the additive transformation \eqref{e:indexchangeaffweyl} for a multiplicative one, and allows one to partially sum the Kirillov formula using the Poisson summation formula. Following this line of reasoning lead, in \cite{LIndexFormula}, to a very general Atiyah-Bott-Segal-Singer fixed-point formula for the index of suitable K-classes on Hamiltonian loop group spaces. In Section 6 we apply this general formula to $\sf{E}^\Sigma V$ and the other Atiyah-Bott classes, and compute the fixed-point contributions. The final result for the class $\sf{E}^\Sigma V$ is the following. Note that in the statement, the Lie algebra $\g$ is identified with $\g^*$ using the basic inner product, $G$ is identified with the diagonal in $\sul{G}=G^b$, $\index^k_{\sul{G}}(\exp(t\sf{E}^\Sigma V))$ is viewed as a formal series of invariant distributions on $\sul{G}$ and we take its pairing with a class function.
\begin{theorem}
\label{t:indexABintroduction}
Let $f \in R(\sul{G})$. Then
\[ \pair{\index^k_{\sul{G}}(\exp(t\sf{E}^\Sigma V))}{f}=\sum_{g \in T_\ell^\reg/W} \Big(\frac{(-1)^{|\mf{R}_+|}J(g_t)^2}{|T_\ell|\det(1-t\ell^{-1}H_V(g_t))}\Big)^{1-\sg}f(g_t) \in \bC\ft,\]
where $\ell=k+\hvee$ and $\hvee$ is the dual Coxeter number of $\g$, $T_\ell=\ell^{-1}\Lambda/\Pi\subset T$ (a finite subgroup), $J$ is the Weyl denominator, $g_t \in \J^\infty_g(T_\bC)$ is the unique $\infty$-jet of a curve (or $\bC\ft$-point) based at $g\in T_\bC$ satisfying $g_t \exp(t\ell^{-1}v(g_t))=g$, $v=\frac{1}{4\pi^2}\nabla \Tr_V$, and $H_V=\d v=\frac{1}{4\pi^2}\tn{Hessian}(\Tr_V)$.
\end{theorem}
There is a more general formula of the same kind involving the other Atiyah-Bott classes $\sf{E}^NV$, and similar formulas with multiple parameters $t_1,...,t_k$ suitable for studying products $\sf{E}^\Sigma V_1 \cdots \sf{E}^\Sigma V_k \in K^\ad_{\sul{G}}(\M_\Sigma)$. As a corollary of Theorem \ref{t:indexABintroduction} and equation \eqref{e:indexreduced} we obtain formulas for K-theory intersection pairings on symplectic quotients. For the simplest example going beyond the case relevant to the Verlinde formulas ($\sf{E}=1$), we deduce that if $k\gg 0$ and $\ul{a}$ is as in \eqref{e:indexreduced} then
\[ \index(\sf{E}_{k\ul{a}})=\frac{\d}{\d t}\bigg|_{t=0}\sum_{g \in T_\ell^\reg/W} \Big(\frac{(-1)^{|\mf{R}_+|}J(g_t)^2}{|T_\ell|\det(1-t\ell^{-1}H_V(g_t))}\Big)^{1-\sg}\Tr_{V_{k\ul{a}}^*}(g_t),\]
and for small $k$ the left hand side is the value of the unique quasi-polynomial function of $k$ that agrees with the right hand side for $k\gg 0$. In the final section we briefly discuss (following \cite{TelemanWoodward}) another example of a generating series, motivated by its appearance in the equivariant Verlinde formula for Higgs bundles \cite{EquivVerlinde, EquivVerlindePf1, EquivVerlindePf2}. 

Taking $f=1$ in Theorem \ref{t:indexABintroduction}, one obtains an expression matching the index theorem of Teleman and Woodward on the moduli stack $\mf{M}_{\bSigma}$ of holomorphic $G_\bC$-bundles (with $\bSigma$ obtained by capping off the boundary of $\Sigma$), a remarkable generalization of the Verlinde formula (recovered by setting $t=0$) that was conjectured by Teleman in \cite{TelemanDefVerlinde} and proved in \cite{TelemanWoodward}. The latter theorem and Theorem \ref{t:indexABintroduction} apply to different contexts and do not admit a direct comparison. However both theorems have formulas for K-theory intersection pairings on symplectic quotients as a common consequence. The basic reason one expects some relationship between the two is that the homotopy type of the stack $\mf{M}_{\bSigma}$ is the homotopy quotient $\M_{\Sigma,\sul{G}}=E\sul{G}\times_{\sul{G}}\M_{\Sigma}$ (cf. \cite{donaldson1992boundary, WoodwardHeatFlow}, \cite[Appendix]{TelemanWoodward} for further discussion). In their proof, Teleman and Woodward introduce (\cite[Section 3]{TelemanWoodward}) a formal series with a twisted affine Weyl antisymmetry property, and it seems likely to us that this series matches our $\sul{T}$-equivariant analytic index. The proof of \eqref{e:indexchangeaffweyl} then suggests a very different perspective on the twisted equivariance property of the former. We also remark that recently an elegant algebraic-geometric approach to these K-theory intersection pairings was given by Trapeznikova \cite{trapeznikova2022}.

We describe an example where the index can be easily calculated directly without the tools developed in \cite{LIndexFormula}. In the discussion below $\g$ is identified with  $\g^*$ using the basic inner product. Let $\Sigma=\bD$ be the disk and let $\M_{\bD}$ denote the moduli space of flat $G$-connections on $\bD$ modulo gauge transformations equal to the identity along $\partial \bD=S^1$. Then $\M_{\bD}$ is an infinite dimensional K{\"a}hler manifold and carries a Hamiltonian action of the loop group $LG$, with moment map
\[ \mu_{\bD}\colon \M_\bD \rightarrow \Omega^1(S^1,\g) \]
induced by pullback of connections to $\partial \bD=S^1$. In this case $\mu_{\bD}$ is an embedding and identifies $\M_\bD$ with the $LG$ orbit of $0 \in \Omega^1(S^1,\g)$ under the gauge action. Since the stabilizer of $0$ is $G\subset LG$ (constant loops), one also has $\M_\bD\simeq LG/G$.

The manifold $X$ is homotopic to the discrete space
\[ \M_\bD \cap \t=W_{\tn{af}}\cdot 0=\Pi, \qquad \Pi=\ker(\exp_T)\subset \t,\]
where on the left hand side we identify $\M_\bD$ with its image in $\Omega^1(S^1,\g)$ and $\t$ with the constant connections $\xi \d \theta$, $\xi \in \t$. The intersection $\M_\bD \cap \t$ also coincides with the fixed point set of the $T\times S^1_{\tn{rot}}$ action that one has on $\M_\bD$, where $S^1_{\tn{rot}}$ rotates the disk. 

The prequantum line bundle may be described as
\[ L=\wh{LG}\times_{\wh{G}} \bC_{(0,-1)} \rightarrow LG/G\simeq \M_\bD,\]
where $\wh{LG}$ is the minimal central extension of the loop group. The spinor bundle $\scr{S}=L^{\hvee}$, and the correction $S$ of $\scr{S}$ to account for non-transversality of the intersection $\M_\bD \cap \t$ is $S=\scr{S}\wh{\otimes}\wedge \n^- \otimes \bC_\rho$ where $\n^-$ is the sum of the negative root spaces and $\rho$ is the half sum of the positive roots. $T$ acts on the fiber of $L$ over $\eta \in \M_\bD\cap \t=\Pi$ by the character $e^\eta$.

Since $X$ is homotopic to a discrete space, the $T$-equivariant index amounts to a sum over the points $\eta \in \Pi$ of the character for the $T$ action on $(S\otimes L^k)|_{\eta}$:
\[ \index^k_T(1)=J\sum_{\eta \in \Pi} e^{\ell \eta}=\sum_{g \in T_\ell^{\tn{reg}}/W} \frac{J(g)}{|T_\ell|}\sum_{w \in W}(-1)^{l(w)}\delta_{wg} \in \D'(T),\]
where $\ell=k+\hvee$, $T_\ell=\ell^{-1}\Lambda/\Pi\subset \t/\Pi=T$. We used the Poisson summation formula and the antisymmetry of $J$. The sum over $T_\ell^{\tn{reg}}/W$ may be taken over any set of representatives.

We next bring in a non-trivial K-class. Fix a finite dimensional representation $V$ of $G$. Over $\A_{\bD}^\flat$ there is a tautological family of first order elliptic operators, where to $A \in \A_{\bD}^\flat$ we assign the $\ol{\partial}$-operator on $\bD$ coupled to the trivial bundle $\bD \times V$ with the vector potential $A$. Imposing the boundary condition that the restriction of a section to $\partial \bD$ has vanishing non-negative Fourier modes, we obtain a family of Fredholm operators over $\A_{\bD}^\flat$ that descends to a $G$-equivariant family of Fredholm operators over $\M_\bD$ representing the K-class $\sf{E}^\bD V \in K_G(\M_\bD)$.

Let $\Tr_V=\sum_\lambda n_\lambda e^\lambda$ denote the restriction of the character to the maximal torus. The pullback of $\sf{E}^\bD V$ to a point $\eta \in \Pi=\M_\bD \cap \t$ is 
\[ \sf{E}^\bD V|_{\eta}=(2\pi \i)^{-1}\partial_\eta \Tr_V=\sum_\lambda n_\lambda \pair{\lambda}{\eta}e^\lambda \in R(T)=K^0_T(\pt).\] 
(One method to verify this is to note $\sf{E}^\bD V|_0=0$ and then use the change of boundary condition formula that we shall see in Section \ref{s:affineWeylaction}.) Therefore the index is
\[ \index_T^k(\sf{E}^\bD V)=J\sum_{\eta \in \Pi}\sum_\lambda n_\lambda \pair{\lambda}{\eta}e^{\ell \eta+\lambda} \in R^{-\infty}(T).\]
Considering instead the formal series $\exp(t\sf{E}^\bD V)\in \bC K_T(\M_\bD)\ft$ we obtain
\[ \index_T^k(\exp(t\sf{E}^{\bD}V))=J\sum_{\eta \in \Pi}e^{\ell \eta+t\sum_\lambda n_\lambda \pair{\lambda}{\eta}e^\lambda} \in \bC R^{-\infty}(T)\ft, \]
which can be simplified, after a formal change of variables, using the Poisson summation formula (see \cite[Section 4.4]{LIndexFormula} for details). The result is
\begin{equation}
\label{e:bdV}
\index_T^k(\exp(t\sf{E}^\bD V))=\sum_{g \in T_\ell^{\tn{reg}}/W}\frac{J(g_t)\sum_{w \in W}(-1)^{l(w)}\delta_{wg_t}}{|T_\ell|\det(1-t\ell^{-1}H_V(g_t))}
\end{equation} 
where $g_t \in \J^\infty_g(T_\bC)$ is the unique $\infty$-jet of a curve (or $\bC\ft$-point) based at $g\in T_\bC$ satisfying $g_t \exp(t\ell^{-1}v(g_t))=g$, $v=\frac{1}{4\pi^2}\nabla \Tr_V$, and $H_V=\d v=\frac{1}{4\pi^2}\tn{Hessian}(\Tr_V)$. On applying induction, equation \eqref{e:bdV} becomes the genus $\rm{g}=0$ case of Theorem \ref{t:indexABintroduction}.

We remark that the $T\times S^1_{\tn{rot}}$-equivariant K-theory of $\M_\bD$ was computed by Kostant and Kumar \cite{kostant1990t}. An algebraic description of both the $T\times S^1_{\tn{rot}}$ and $T$-equivariant K-theory groups of $\M_\bD$ are given in \cite{lam2010k}. For the $T$-equivariant K-theory the condition is as follows. Since $W_{\tn{af}}=W\ltimes \Pi$, the restriction of an element $\sf{E}\in K_T(\M_\bD)$ to $\M_\bD\cap \t=\Pi$ can be viewed as a function $\psi \colon W_{\tn{af}}\rightarrow R(T)$ constant on $W$ cosets. Then the $R(T)$-module $K_T(\M_\bD)$ is isomorphic to the space of such functions $\psi$ with the additional GKM-type property that for every $d>0$ and root $\alpha$, the function $\mf{d}^d_{\alpha^\vee}\psi$ takes values in $(1-e^\alpha)^d\cdot R(T)$, where $\mf{d}_{\alpha^\vee}$ is the finite difference operator $(\mf{d}_{\alpha^\vee}\psi)(w)=\psi(w)-\psi(\alpha^\vee\cdot w)$. As must be so $\sf{E}^\Sigma V$ satisfies the condition (and trivially for $d>1$ in the sense that $\mf{d}^d_{\alpha^\vee}\psi=0$). $\sf{E}^\Sigma V$ is rather different from the Schubert-type generators in \cite{lam2010k}. We conjecture that the admissible classes are those for which the GKM condition becomes trivial for some finite $d$.

\bigskip

\noindent \textbf{Acknowledgements.} I thank Michele Vergne for discussions and for introducing me to the work of physicists \cite{EquivVerlinde} on the equivariant Verlinde formula several years ago. I thank Eckhard Meinrenken, Nigel Higson, Reyer Sjamaar, Chris Woodward, Maxim Braverman and Daniel Halpern-Leistner for helpful discussions.

\bigskip

\noindent \textbf{Notation.} We often deal with vector spaces/bundles that are $\bZ_2$-graded and in this context $[a,b]$ denotes the graded commutator of linear operators $a,b$. Graded tensor products are denoted $\hotimes$.

Let $G$ be a compact connected simply connected simple Lie group with Lie algebra $\g$. Fix a maximal torus $T \subset G$ with Lie algebra $\t$, and a choice of positive chamber $\t_+\subset \t$. Let $\mf{A}\subset \t_+$ denote the fundamental alcove containing $0$. Let $W$ be the Weyl group, and let $\mf{R} \supset \mf{R}_+$ be the set of roots, positive roots respectively. The integer or co-root lattice is denoted $\Pi=\ker(\exp\colon \t \rightarrow T)\simeq\Hom(U(1),T)$, and the real weight lattice by $\Lambda=\Hom(\Pi,\bZ)\simeq\Hom(T,U(1))$. For $\lambda \in \Lambda$, the corresponding homomorphism $T\rightarrow U(1)$ is denoted $e^\lambda$, and if $\xi \in \t$, $u=\exp(\xi)$ then $e^\lambda(u)=u^\lambda=e^{2\pi \i\pair{\xi}{\lambda}}$. The lattice $\Pi$ in $\t$ determines a canonical normalization of Lebesgue measure, that we use to identify distributions with generalized functions. Usually $u,g$ denote elements of $T$ (or $T_\bC$), $\xi$ an element of $\t$ (or $\t_\bC$), $\eta$ an element of $\Pi$, and $\lambda$ an element of $\Lambda$. 

The representation ring of $T$ is $R(T)=\bZ[\Lambda]$, and $\bC R(T)=\bC\otimes R(T)$. The formal completion $R^{-\infty}(T)=\bZ^\Lambda$ is an $R(T)$-module. If $T$ acts on a Hilbert space $H$, the isotypical subspace of weight $\lambda \in \Lambda$ is denoted $H_{[\lambda]}$. If $a$ is a $T$-equivariant linear operator on $H$, then the induced operator on $H_{[\lambda]}$ is denoted $a_{[\lambda]}$.

For a simple Lie algebra, every invariant integral inner product is a positive integer multiple of the basic inner product `$\cdot$', normalized such that the length of the short co-roots is $\sqrt{2}$. For $\mf{su}(n)$ the basic inner product is
\[ (\xi_1,\xi_2)\mapsto \xi_1\cdot\xi_2=-\frac{1}{4\pi^2}\Tr_{\bC^n}(\xi_1\xi_2).\] 
We use the basic inner product to identity $\g\simeq \g^*$ throughout. In particular $\Lambda$ becomes a lattice in $\t$ and for $\ell \in \bZ_{>0}$, we can define a finite subgroup $T_\ell:=\ell^{-1}\Lambda/\Pi \subset \t/\Pi=T$.

Let $\hvee$ be the dual Coxeter number of $\g$. For example if $\g=\mf{su}(n)$ then $\hvee=n$. Beginning in Section 4 of the article, $k$ denotes an integer strictly greater than $-\hvee$, and $\ell=k+\hvee$.

Let $b \ge 1$ be an integer. Objects associated to the $b$-fold product $\sul{G}:=G^b$ will be denoted with an underline. For example, $\sul{T}=T^b\subset G^b$, $\ul{\Pi}=\Pi^b\subset \ul{\t}=\t^b$, and so on.

\section{Framed moduli spaces of flat connections}

\subsection{Loop groups}
For background on loop groups see \cite{PressleySegal}. Let $LG$ denote the space of maps $S^1=\bR/\bZ \rightarrow G$ of fixed Sobolev class $\varsigma>\frac{3}{2}$. This is a Banach Lie group (the group operation is pointwise multiplication of loops) with Lie algebra $L\g$ identified with $\g$-valued $0$-forms $\Omega^0_\varsigma(S^1,\g)$ of Sobolev class $\varsigma$.

The $U(1)$ central extensions of the loop group are classified by $\ell \in \bZ$, where the corresponding central extension $\wh{LG}^{(\ell)}$ has Lie algebra cocycle
\[ (\xi_1,\xi_2)\mapsto 2\pi \i \int_{S^1} \ell\d\xi_1\cdot \xi_2.\] 
The group $G$ sits inside $LG$ as the subgroup of constant loops. The central extension $\wh{LG}^{(\ell)}$ trivializes over $G$, hence $G$ can be regarded as a subgroup of $\wh{LG}^{(\ell)}$.

Let $L\g^*=\Omega^1_{\varsigma-1}(S^1,\g)$ denote the space of $\g$-valued $1$-forms (or connections) of Sobolev class $\varsigma-1$. Since $\varsigma-1>-\varsigma$, there is a non-degenerate pairing $L\g \times L\g^*\rightarrow \bR$ given in terms of the inner product $\cdot$ on $\g$ followed by integration over $S^1$. The loop group $LG$ acts smoothly and properly on $L\g^*$ by gauge transformations:
\begin{equation} 
\label{e:gauge}
a\mapsto g\cdot a=\Ad_g(a)-\d gg^{-1}, \qquad g \in LG, a \in L\g^*.
\end{equation}
Given $\xi \in L\g$ the corresponding vector field for the infinitesimal action on $L\g^*$ is $\xi_{L\g^*}(a)=\partial_u|_{u=0} \exp(-u\xi)\cdot a=\d_a \xi$ at $a \in L\g^*$, where $\d_a=\d+\ad_a$ is the covariant derivative defined by the connection $a$. The subgroup $\Omega G \subset LG$ consisting of loops beginning at $1 \in G$ acts on $L\g^*$ freely and properly with quotient $L\g^*/\Omega G \simeq G$, and the $LG$-action on $L\g^*$ descends to the action of $G$ on itself by conjugation. The holonomy map
\[ \hol \colon L\g^* \rightarrow G \]
is the quotient map for the $\Omega G$ action.

\subsection{Framed moduli spaces of flat connections.}\label{s:modsp}
Let $\Sigma$ be a compact connected oriented genus $\rm{g}$ surface with $\partial \Sigma \ne \emptyset$. We label the boundary components $1,...,b$ ($b$ connected components) and denote the $j$-th boundary component by $\partial_j \Sigma$. We orient the boundary such that if $v$ is a non-vanishing interior-pointing vector field along the boundary, then $\iota(v)\tn{Vol}_\Sigma|_{\partial \Sigma}$ is a volume form on the boundary (with this orientation a minus sign appears in Stokes' theorem). We also fix orientation-preserving parametrizations of the boundary components $\partial_j\Sigma\simeq S^1=\bR/\bZ$. Recall $\sul{G}=G^b$, which we view as the set of locally constant maps $\partial \Sigma \rightarrow G$. The loop group $L\sul{G}=LG^b$ is identified with the set of maps $\partial \Sigma \rightarrow G$ of Sobolev class $\varsigma$.

As $G$ is simply connected, any principal $G$-bundle over $\Sigma$ is topologically trivial. Let $\G_\Sigma$ be the gauge group, that is, the set of maps $\Sigma \rightarrow G$ of Sobolev class $\varsigma+\frac{1}{2}$. There is a continuous group homomorphism 
\[ \scr{R} \colon \G_\Sigma\rightarrow L\sul{G} \] 
given by restriction to the boundary. Since $G$ is simply connected, $\scr{R}$ is surjective, hence we obtain a short exact sequence
\[ 1 \rightarrow \G_{\Sigma,\partial \Sigma}\rightarrow \G_\Sigma \rightarrow L\sul{G} \rightarrow 1.\]
Let
\[ \A_\Sigma=\Omega^1_{\varsigma-\frac{1}{2}}(\Sigma,\g) \supset \A_\Sigma^\flat \]
be the space of connections, resp. the space of flat connections, on the principal $G$-bundle $\Sigma \times G$, of Sobolev class $\varsigma-\frac{1}{2}$. The affine space $\A_\Sigma$ carries the well-known Atiyah-Bott symplectic structure \cite{AtiyahBottYangMills}, and the action of $\G_{\Sigma,\partial \Sigma}$ is Hamiltonian for the moment map $A \mapsto \tn{curv}(A)$. 
\begin{definition}
The symplectic quotient
\[ \M_\Sigma=\A_\Sigma^\flat/\G_{\Sigma,\partial \Sigma} \]
is an infinite dimensional weakly symplectic Banach manifold, equipped with a Hamiltonian action of the loop group $L\sul{G}$, with proper moment map
\[ \mu_{\M_\Sigma}\colon \M_\Sigma \rightarrow \Omega^1_{\varsigma-1}(\partial \Sigma,\g)=L\ul{\g}^* \]
induced by pullback of the connection to the boundary. 
\end{definition}
$\M_\Sigma$ is the moduli space of flat connections on $\Sigma$ with framing along the boundary; it is an example of a proper Hamiltonian $L\sul{G}$-space, cf. \cite{MWVerlindeFactorization, meinrenken1999cobordism}. We review the construction of Banach manifold charts on $\M_\Sigma$ in the next section. The subgroup $\Omega \sul{G} \subset L\sul{G}$ of based loops acts freely and properly on $\M_\Sigma$.
\begin{definition}
The quotient
\[ M_\Sigma=\M_\Sigma/\Omega \sul{G} \]
is a smooth finite dimensional compact $\sul{G}$-space, that we refer to as the \emph{holonomy manifold} of $\M_\Sigma$. The quotient map $\M_\Sigma \rightarrow M_\Sigma$ will be denoted $\hol_{\M_\Sigma}$.
\end{definition}
There is a pullback diagram
\begin{equation} 
\label{eqn:Pullback}
\xymatrixcolsep{4pc}
\xymatrix{
\M_\Sigma \ar[r]^{\mu_{\M_\Sigma}} \ar[d]^{\hol_{\M_\Sigma}} & L\ul{\g}^* \ar[d]^{\hol} \\
M_\Sigma \ar[r] & \sul{G}
}
\end{equation}
Thus $M_\Sigma$ is the moduli space of flat connections on $\Sigma$ with framing at the set of points 
\[ P=\{p_1,...,p_b\} \subset \partial \Sigma \] 
given by the images of $0 \in S^1=\bR/\bZ$ under the boundary parametrizations. Since the gauge equivalence class of a flat connection can be described in terms of its holonomy data, there is an identification
\[ M_\Sigma\simeq \Hom(\Pi_1(\Sigma,P),G),\]
where $\Pi_1(\Sigma,P)$ is the fundamental groupoid of $\Sigma$ based at the set of points $P$. There is a residual action of $\sul{G}$ on $M_\Sigma$ changing the framing at the points $P$ and one has canonical identifications of the quotients
\begin{equation} 
\label{e:quotientsidentify}
\A_\Sigma^\flat/\G_\Sigma\simeq \M_\Sigma/L\sul{G}\simeq M_\Sigma/\sul{G} \simeq \Hom(\pi_1(\Sigma,p_1),G)/G.
\end{equation}
The manifold $M_\Sigma$ is an example of a quasi-Hamiltonian $\sul{G}$-space \cite{AlekseevMalkinMeinrenken}, although we shall not use that perspective in this article.

Later on we shall also need a slightly different presentation of $\M_\Sigma$. Let $\C$ be a collar neighborhood of $\partial \Sigma$. Let $\A_{\Sigma,\C}^\flat$ denote flat connections that are constant on $\C$ equal to the pullback of their pullback to the boundary, and let $\G_{\Sigma,\C}$ denote gauge transformations that are equal to $1$ on $\C$. Then one has
\begin{equation} 
\label{e:trivialcollar}
\M_\Sigma \simeq \A_{\Sigma,\C}^\flat/\G_{\Sigma,\C}.
\end{equation}
One can prove this using \eqref{e:quotientsidentify}---see for example the proof of \cite[Theorem 3.2]{MWVerlindeFactorization} for a similar argument.

Let $\ul{a}=(a_1,...,a_b) \in \ul{\mf{A}}=\mf{A}^b$. The symplectic quotient 
\[ \M_{\Sigma,\ul{a}}=\mu_{\M_\Sigma}^{-1}(\ul{a})/(L\sul{G})_{\ul{a}} \]
is the moduli space of flat connections on $\Sigma$ with holonomy around the $j$-th boundary component in the conjugacy class of $\exp(a_j)$.

\subsection{Coulomb gauge charts}
Manifold charts for $\M_\Sigma$ are constructed using the implicit function theorem for Banach spaces (see for example \cite[p.103]{donaldson1992boundary}, \cite[Section 3.1]{MWVerlindeFactorization}, \cite[Appendix C]{WoodwardHeatFlow}, \cite[Chapters 2, 4]{DonaldsonKronheimer}, \cite{atiyah1978self, wilkins1989slice}). We review this construction since we will make use of it later. 

Use a faithful representation $\bC^N$ to view $G$ as a matrix group. Note that for $A \in \A_\Sigma$, $g \in \G_\Sigma$
\begin{equation} 
\label{e:Aflat}
A-g\cdot A=(\d_A g)g^{-1},
\end{equation}
where $\d_Ag=\d g+[A,g]$ is the covariant derivative for the representation of $\g$ on $\End(\bC^N)$. If $g$ fixes $A$ then \eqref{e:Aflat} reads $\d_A g=0$, and thus by parallel translation $g$ is uniquely determined by its value at a single point of the connected manifold $\Sigma$. In particular if $g \in \G_{\Sigma,\partial \Sigma}$, then $g|_{\partial \Sigma}=1$ and uniqueness implies $g=1$ identically. This shows that the action of $\G_{\Sigma,\partial \Sigma}$ on $\A_\Sigma$ is free. At the level of Lie algebras, this means that the map
\begin{equation}
\label{e:injectiveorbit}
\gamma \in \Omega^0_{\varsigma+\frac{1}{2},\partial}(\Sigma,\g)\mapsto \d_A \gamma \in \Omega^1_{\varsigma-\frac{1}{2}}(\Sigma,\g)
\end{equation}
has trivial kernel, where the subscript `$\partial$' indicates that we impose the boundary condition $\gamma|_{\partial \Sigma}=0$. Equation \eqref{e:Aflat} further shows that if $g \in \G_{\Sigma,\partial \Sigma}$ and $\|A-g\cdot A\|_{\varsigma-\frac{1}{2}}$ is small, then $g$ is close to $1 \in \G_{\Sigma,\partial \Sigma}$ in the Sobolev $\varsigma+\frac{1}{2}$ topology (here we use ellipticity of $\d_A$ on $0$-forms and compactness of $\Sigma$).

Fix a Riemannian metric on $\Sigma$. Let $A \in \A$. We say that $A'$ is in \emph{Coulomb gauge relative to }$A$ if
\[ \d_A^*(A'-A)=0.\]
Define Coulomb gauge slices
\[ C_{A,\epsilon}=\{A'\in \A_\Sigma|\d_A^*(A'-A)=0,\|A'-A\|_{\varsigma-\frac{1}{2}}<\epsilon\}.\]
Note that these slices are equivariant for the full gauge group $\G_\Sigma$ in the sense that
\begin{equation}
\label{e:Couequiv}
C_{g\cdot A,\epsilon}=g\cdot C_{A,\epsilon}, \qquad g \in \G_\Sigma.
\end{equation} 
For $A \in \A_\Sigma^\flat$ let 
\[U_{A,\epsilon}=\{A' \in \A|\tn{curv}(A')=0, \d_A^*(A'-A)=0, \|A'-A\|_{\varsigma-\frac{1}{2}}<\epsilon\} \subset C_{A,\epsilon}.\]
\begin{proposition}
\label{p:charts}
Let $A \in \A_\Sigma$. There is an $\epsilon>0$ sufficiently small (depending on $A$) such that if $\|A'-A\|_{\varsigma-\frac{1}{2}}<\epsilon$ then there is a unique $g_A(A')\in \G_{\Sigma,\partial \Sigma}$ such that $g_A(A')\cdot A' \in C_{A,\epsilon}$, and moreover $g_A(A')$ depends smoothly on $A'$. Consequently $C_{A,\epsilon}$ is a slice for the $\G_{\Sigma,\partial \Sigma}$ action and $q \colon \A_\Sigma\rightarrow \A_\Sigma/\G_{\Sigma,\partial \Sigma}$ is a principal $\G_{\Sigma,\partial \Sigma}$-bundle. For $A \in \A_\Sigma^\flat$ and $\epsilon>0$ sufficiently small, there is a diffeomorphism from $U_{A,\epsilon}$ to its tangent space at $A$, $T_AU_{A,\epsilon}$, and the composition of the latter with the quotient map $q$ provides a local chart for $\M_\Sigma$ near $q(A)$.
\end{proposition}
\begin{proof}
Let $\alpha=A'-A$. By assumption $\|\alpha\|_{\varsigma-\frac{1}{2}}<\epsilon$. Consider the function of $g \in \G_{\Sigma,\partial \Sigma}$, $\alpha \in \Omega^1_{\varsigma-\frac{1}{2}}(\Sigma,\g)$ defined by
\[ F(g,\alpha)=\d_A^*(g\cdot (A+\alpha)-A).\]
Clearly $F(1,0)=0$. We claim that for $\epsilon$ sufficiently small, the equation $F(g,\alpha)=0$ (with $\alpha$ fixed) has a unique solution $g \in \G_{\Sigma,\partial \Sigma}$ such that $\|g\cdot (A+\alpha)-A\|_{\varsigma-\frac{1}{2}}<\epsilon$. Since $\alpha$, $g\cdot (A+\alpha)-A$ are small in the Sobolev $\varsigma-\frac{1}{2}$ norm by assumption, the discussion following equation \eqref{e:Aflat} implies that $g$ is close to $1$ in $\G_{\Sigma,\partial \Sigma}$ in the Sobolev $\varsigma+\frac{1}{2}$ topology. Thus taking $\epsilon$ sufficiently small we can assume that we are in the region where the Banach space implicit function theorem applies. The first partial derivative of $F$ with respect to $g$ at the point $(1,0)$ is
\[ (\bm{\d}_1F)_{(1,0)}(\gamma)=-\d_A^*\d_A\gamma, \qquad \gamma \in \Omega^0_{\varsigma+\frac{1}{2},\partial}(\Sigma,\g).\]
Since the Dirichlet problem for the Laplace-type operator $\d_A^*\d_A$ is uniquely soluble (this follows from, for example, the Lax-Milgram theorem, using injectivity of \eqref{e:injectiveorbit} to check coercivity), the implicit function theorem yields existence and uniqueness of $g$ and shows moreover that $g=:g_A(A')$ depends smoothly on $A'=A+\alpha$.

The flow out $\G_{\Sigma,\partial \Sigma}\cdot C_{A,\epsilon}$ is an open subset of $\A$. It follows from the uniqueness proved above that every $A' \in \G_{\Sigma,\partial \Sigma}\cdot C_{A,\epsilon}$ is in the $\G_{\Sigma,\partial \Sigma}$ orbit of a unique element of $C_{A,\epsilon}$. Thus $C_{A,\epsilon}$ is a slice for the free $\G_{\Sigma,\partial \Sigma}$ action, providing a local section for $q$. A similar argument using the implicit function theorem shows that given two local sections $C_{A_i,\epsilon}$, $i=1,2$ with $\epsilon$ sufficiently small, the transition function $q(C_{A_1,\epsilon})\cap q(C_{A_2,\epsilon}) \rightarrow \G_{\Sigma,\partial \Sigma}$ is smooth.

Next let $A \in \A_\Sigma^\flat$. The derivative of $\tn{curv}(\cdot)$ at $A$ is 
\[ (\bm{\d}\tn{curv})_A(\alpha)=\d_A\alpha.\]
The map 
\[ \alpha \in \Omega^1_{\varsigma-\frac{1}{2}}(\Sigma,\g)\cap \ker(\d_A^*)\mapsto \d_A\alpha \in \Omega^2_{\varsigma-\frac{3}{2}}(\Sigma,\g) \] 
is surjective: $\d_A\alpha=\beta$ has a solution $\alpha=\ast \d_A \gamma$ where $\gamma$ is the unique solution of the Dirichlet problem
\[ -\d_A^*\d_A\gamma=\ast \beta, \qquad \gamma|_{\partial \Sigma}=0,\]
and moreover $\d_A^*\alpha=0$ since $\d_A^2=0$ by flatness of $A$. This shows that $0$ is a regular value of $\tn{curv}$. By the implicit function theorem, for $\epsilon$ sufficiently small, $U_{A,\epsilon}$ is diffeomorphic to its tangent space at $A$
\[ T_AU_{A,\epsilon}=\{\alpha \in \Omega^1_{\varsigma-\frac{1}{2}}(\Sigma,\g)|\d_A\alpha=0=\d_A^*\alpha\}.\]
That this yields a local chart for $\M_\Sigma$ near $q(A)$ follows.
\end{proof}

\section{Atiyah-Bott classes for framed moduli spaces of flat connections}\label{s:ABclasses}
In this section we will describe various equivariant K-theory classes of $\M_\Sigma$. These classes generalize K-theory classes introduced by Atiyah and Bott \cite{AtiyahBottYangMills} in case of a closed surface, and so we refer to these classes as Atiyah-Bott classes.

By $\sul{G}$-equivariant K-theory of a $\sul{G}$-space $Y$ we mean the $\sul{G}$-equivariant non-compactly supported K-theory, which we denote $K_{\sul{G}}(Y)=K^0_{\sul{G}}(Y)$: homotopy classes of $\sul{G}$-equivariant continuous maps from $Y$ to the space of Fredholm operators on $L^2(\sul{G})\otimes \ell^2(\bZ)$ equipped with the norm topology (cf. \cite[Section 5]{segal1970fredholm}). In this article it is convenient to represent K-theory classes by pairs $(\E,Q)$ where $\E \rightarrow Y$ is a $\bZ_2$-graded $\sul{G}$-equivariant Hilbert bundle (with structure group the unitary group carrying the norm topology), and $Q$ is a family of possibly unbounded odd self-adjoint Fredholm operators on the fibers of $\E$ such that the corresponding bounded family $F=Q(1+Q^2)^{-1/2}$ is norm continuous. Given such a pair $(\E,Q)$, one obtains a continuous map from $Y$ to the space of Fredholm operators by trivializing $\E$ (after stabilization if necessary) and taking the component $F^+$ of $F=Q(1+Q^2)^{-1/2}$ that maps $\E^+$ to $\E^-$. 

Fix a finite dimensional complex representation $(V,\pi_V)$ of $G$. Let $g\in \G_\Sigma$ act on the trivial bundle $\Sigma \times V$ by 
\[ g\cdot (p,v)=(p,\pi_V(g(p))v),\qquad p \in \Sigma, v \in V.\]
\begin{definition}
\label{d:EV}
The quotient of $\A_\Sigma^\flat\times (\Sigma \times V)$ by the diagonal action of $\G_{\Sigma,\partial \Sigma}$ is an $L\sul{G}$-equivariant vector bundle
\[ \sf{E}V=(\A_\Sigma^\flat \times \Sigma \times V)/\G_{\Sigma,\partial \Sigma} \rightarrow  (\A_\Sigma^\flat\times \Sigma)/\G_{\Sigma,\partial \Sigma}=\M \times \Sigma.\]
It defines a class $\sf{E}V\in K_{\sul{G}}^0(\M\times \Sigma)$.
\end{definition}

The strategy is to construct K-theory classes on $\M_\Sigma$ by pulling $\sf{E}V$ back to $\M \times N$ where $\iota \colon N \hookrightarrow \Sigma$ is a submanifold, and then push forward along $p\colon \M \times N \rightarrow \M$. The class so obtained only depends on the homotopy class of the map $\iota$. This works directly when $\dim(N)=0,1$ and, suitably interpreted, when $\dim(N)=2$ (i.e. $N=\Sigma$).

\subsection{Atiyah-Bott classes for $\dim(N)=0,1$}\label{s:ABclasses}
The simplest case is $N=\pt$. Since $\Sigma$ is connected, any two points in $\Sigma$ are connected by a path, and hence we obtain only one distinct K-theory class $\sf{E}^{\pt}V \in K^0_{\sul{G}}(\M_\Sigma)$. Choosing the point to lie on $\partial \Sigma$, we see that $\sf{E}^{\pt}V$ is represented by the trivial vector bundle $\M_\Sigma \times V$, where all but one of the copies of $G$ in $\sul{G}=G^b$ acts trivially on the fibers (in particular, choosing different factors in $G^b$ to act on the fiber $V$ leads to homotopic $G^b$-equivariant vector bundles).

Moving to dimension $1$, let $\iota \colon N=C\hookrightarrow \Sigma$ be a closed curve. Pulling $\sf{E}V$ back to $\M_\Sigma \times C$ and taking $L^2$-sections along $C$ yields an $L\sul{G}$-equivariant Hilbert bundle:
\[ \E_C=(\A_\Sigma^\flat\times L^2(C,V))/\G_{\Sigma,\partial \Sigma} \rightarrow \A_\Sigma^\flat/\G_{\Sigma,\partial \Sigma}=\M_\Sigma \]
where $\G_\Sigma$ acts diagonally, with the action on $L^2(C,V)$ being through evaluation at points along $C$. (Note that the gauge action preserves the $L^2$-inner product.) Fix a parametrization $C\simeq S^1$ and let $\DiracC_0=\frac{1}{\i} \frac{\d}{\d s}$ acting on smooth sections of the trivial bundle $C\times V$ over $C$. More generally for $A \in \A_\Sigma^\flat$, let 
\[ \DiracC_A=\DiracC_0+\iota^*A.\] 
This defines a tautological family of differential operators 
\[ \DiracC_\bullet=\{\DiracC_A\mid A \in \A_\Sigma^\flat\} \] 
over $\A_\Sigma^\flat$, which is equivariant for the diagonal action of $\G_\Sigma$ on $\A_\Sigma^\flat\times L^2(C,V)$. Descending the family along the fibers of $q \colon \A_\Sigma^\flat\rightarrow \A_\Sigma^\flat/\G_{\Sigma,\partial \Sigma}=\M_\Sigma$ we obtain an $L\sul{G}$-equivariant family of Dirac operators, denoted $\DiracC_{q(\bullet)}$, acting on the fibers of the Hilbert bundle $\E_C$. This family defines an odd K-theory class $\sf{E}^CV \in K^1_{\sul{G}}(\M_\Sigma)$. Even classes can be obtained by taking products of even numbers of these classes. For example if $C_1,C_2$ are two closed curves, then $\sf{E}^{C_1}V_1\cdot \sf{E}^{C_2}V_2 \in K^0_{\sul{G}}(\M_\Sigma)$ is represented by a family of elliptic differential operators on the 2-torus $C_1 \times C_2 \subset \Sigma \times \Sigma$. 

\begin{proposition}
\label{p:oddclasshomotopy}
$\sf{E}^CV \in K^1_{N(\sul{T})}(X)$ only depends on the free homotopy class of $C$ in the closed Riemann surface $\bSigma$ obtained by gluing a disk to each component of the boundary of $\Sigma$.
\end{proposition}
\begin{proof}
By our remarks above, it is clear that $\sf{E}^CV$ only depends on the free homotopy class of $C$ in $\Sigma$. Suppose $C=\partial_j\Sigma$ is one of the boundary components of $\Sigma$. Then $\sf{E}^C V \in K^1_{\sul{G}}(\M_\Sigma)$ is the pullback of a tautological $G$-equivariant class over $L\g^*$. But $K^1_G(L\g^*)=0$ because $L\g^*$ is $G$-equivariantly contractible (by solving the parallel transport equation for a connection, $L\g^*$ is the same as the space $P_1 G$ of paths of Sobolev class $\varsigma$ in $G$ beginning at $1$). It follows from this that the K-theory class only depends on the free homotopy class of $C$ in $\bSigma$.
\end{proof}
 
The 2-dimensional case $N=\Sigma$ is less immediate. Since $\partial \Sigma\ne \emptyset$, the $\bar{\partial}$-operator on $\Sigma$ is not a Fredholm operator. There are several different approaches one might take to construct a Fredholm family. The method that we pursue below is to impose a boundary condition.
 
\subsection{A family of elliptic boundary problems.}\label{s:familyAPS}
Fix a Riemannian metric on $\Sigma$ of product form in a collar neighborhood of the boundary, where we use the standard metric on $\partial_j\Sigma\simeq S^1=\bR/\bZ$ giving the circle circumference $1$. Choose a spin structure on $\Sigma$ and let $\c^\Sigma \colon \bC l(T\Sigma)\rightarrow \End(S^\Sigma)$, $S^\Sigma=S^{\Sigma,+}\oplus S^{\Sigma,-}$ be the corresponding $\bZ_2$-graded $\bC l(T\Sigma)$-module. Let $\DiracS_0$ be the Dirac operator acting on sections of $S^\Sigma \boxtimes V$ (coupled to $\Sigma \times V$ using the trivial connection). 
\ignore{
\begin{remark}
If we equip $\Sigma$ with a complex structure, then $S^\Sigma=P\otimes \wedge T^*_{0,1}\Sigma$ where $P$ is a choice of holomorphic square root for the canonical bundle, and $\DiracS_0$ is the Hodge-Dolbeault operator $\sqrt{2}(\bar{\partial}^P+(\bar{\partial}^P)^*)$. (An only slightly more complicated index formula would result if we did not twist by $P$.)
\end{remark}
}
For $A \in \A_\Sigma^\flat$, let
\begin{equation}
\label{e:DiracFamily}
\DiracS_A=\DiracS_0+(\c^\Sigma \boxtimes \pi_V)(A)=\matr{0}{\DiracSM_A}{\DiracSP_A}{0},
\end{equation}
be the Dirac operator obtained by coupling the spin Dirac operator on $\Sigma$ to $\Sigma \times V$ using the connection (or vector potential) $A$. This defines a tautological family of Dirac operators $\DiracS_\bullet=(\DiracS_A)_{A \in \A_\Sigma^\flat}$ parametrized by $\A_\Sigma^\flat$.

Owing to the non-empty boundary $\partial \Sigma$, the operators $\DiracS_A$ are not Fredholm. To obtain a Fredholm family we impose a boundary condition. The spin structure $S^\Sigma$ on $\Sigma$ induces a spin structure $S^{\partial \Sigma}=S^{\Sigma,+}|_{\partial \Sigma}$ on $\partial \Sigma$. The general formula for the Clifford action on the boundary is $\c^{\partial \Sigma}(w)=\c^\Sigma(v)^{-1}\c^\Sigma(w)$, where $w \in T\partial \Sigma$ and $v$ is the interior-pointing unit normal vector field along the boundary. In this particularly simple case, note that $(v,\partial_s)$ is an oriented orthonormal frame of $T\Sigma|_{\partial \Sigma}$, hence $\c^{\partial \Sigma}(\partial_s)=-\c^\Sigma(v)\c^\Sigma(\partial_s)$ is $\i=\sqrt{-1}$ times the chirality operator, hence $\c^{\partial \Sigma}(\partial_s)$ acts as multiplication by $\i$ on $S^{\partial \Sigma}$. Let $\partial=\partial_0$ denote the spin Dirac operator on the boundary twisted by $V$, which acts on sections of $S^{\partial \Sigma}\boxtimes V$. Let $a \in L\ul{\g}^*=\Omega^1_{\varsigma-1}(\partial \Sigma,\g)$ be a connection on the boundary. The corresponding spin Dirac operator on the boundary is
\begin{equation} 
\label{e:partiala}
\partial_a=\partial+(\c^{\partial \Sigma}\boxtimes\pi_V)(a).
\end{equation}
For $v \in \bR$ let
\[ B_{< v}(\partial_a)\subset L^{2,\frac{1}{2}}(\partial \Sigma,S^{\partial\Sigma}\boxtimes V), \qquad S^{\partial \Sigma}=S^{\Sigma,+}|_{\partial \Sigma}, \]
be the closure (in the Sobolev $\frac{1}{2}$ norm) of the sum of the eigenspaces of $\partial_a$ with eigenvalue $< v$. We similarly define $B_{\le v}(\partial_a)$, $B_{>v}(\partial_a)$, $B_{\le v}(\partial_a)$. Let 
\[ \scr{R}\colon L^{2,1}(\Sigma,S^{\Sigma,+}\boxtimes V)\rightarrow L^{2,\frac{1}{2}}(\partial\Sigma,S^{\partial\Sigma}\boxtimes V) \]
denote restriction to the boundary.

Imposing the boundary condition $B_{<0}(\partial)$ on the operators of the family $\DiracSP_\bullet$, we obtain a family of unbounded Fredholm operators:
\begin{definition}
\label{d:thefamily}
For $A \in \A_\Sigma^\flat$, let $(\DiracSP_A,B_{<0}(\partial))$ denote the unbounded Hilbert space operator defined by the differential operator $\DiracSP_A$ with domain
\[ \dom(\DiracSP_A,B_{<0}(\partial))=L^{2,1}(\Sigma,S^{\Sigma,+}\boxtimes V)\cap \scr{R}^{-1}(B_{<0}(\partial)).\]
The corresponding odd self-adjoint operator in $L^2(\Sigma,S^\Sigma\boxtimes V)$ is denoted
\[ D_A=\matr{0}{(\DiracSP_A,B_{<0}(\partial))^*}{(\DiracSP_A,B_{<0}(\partial))}{0}.\]
Let $D_\bullet=\{D_A\mid A \in \A_\Sigma^\flat\}$ denote the tautological family. The bounded transforms are denoted
\[ F_A=(1+D_A^2)^{-1/2}D_A, \qquad F_\bullet=(F_A)_{A \in \A_\Sigma^\flat}.\]
\end{definition}
The condition $\scr{R}(s)\in B_{< 0}(\partial)$ is the \emph{Atiyah-Patodi-Singer (APS) boundary condition} \cite{AtiyahPatodiSingerI} for $\DiracSP_0$. The operator $D_A^+=(\DiracSP_A,B_{<0}(\partial))$ is Fredholm, and in fact elliptic in the strong sense that there is an estimate $\|s\|_1\le C(\|D_A^+ s\|+\|s\|)$ for the $L^{2,1}$-norm when $s \in \dom(D_A^+)$; by the Rellich lemma $D_A$ has compact resolvent. A detailed recent reference for the APS and other more general elliptic boundary problems is \cite{BarBallmann, BarBallmannGuide}. More generally one can define Fredholm operators $(\DiracSP_A,B_{<v}(\partial_a))$ for any $v \in \bR$, $a \in L\ul{\g}^*$. The index is independent of $A$, but is highly sensitive to the choice of $v$, $a$.

The index of each operator in the family of Definition \ref{d:thefamily} is $0$. We may check this using the Atiyah-Patodi-Singer theorem. The spin structure on $S^1$ induced by that on $\Sigma$ is the non-trivial one (cf. \cite[Exercise 2.2.18]{DanFreedLecNotes}). The corresponding Dirac operator is (after making a non-canonical choice of trivialization of $S^{\partial \Sigma}$) the operator $\i \frac{\d}{\d s}-\pi$. The latter has spectrum $(2\bZ+1)\pi$, trivial kernel and trivial $\eta$-invariant. Since $\Ahat(\Sigma)=1$, the Atiyah-Patodi-Singer theorem reads
\[ \index(\DiracSP_A,B_{<0}(\partial))=0-\frac{0+0}{2}=0.\]
See also \cite[pp.61--62]{AtiyahPatodiSingerI} for general discussion of the spin case, and also for the discussion of the $\ol{\partial}$-operator on Riemann surfaces.

\subsection{Atiyah-Bott classes for $N=\Sigma$.}\label{s:ABSigma}
The gauge group $\G_\Sigma$ acts on the Hilbert space $\E_0=L^2(\Sigma,S^\Sigma\boxtimes V)$ point-wise via the representation $\pi_V$. Taking the quotient by the diagonal action produces a smooth $L\sul{G}$-equivariant Hilbert bundle
\[ \E=\big(\A_\Sigma^\flat\times \E_0\big)/\G_{\Sigma,\partial \Sigma} \rightarrow \A_\Sigma^\flat/\G_{\Sigma,\partial \Sigma}=\M_\Sigma. \]
Equivalently $\E$ is obtained from the vector bundle $(\bC\boxtimes S^\Sigma)\otimes EV$ over $\M_\Sigma\times \Sigma$ by taking $L^2$-sections along the fibers $\M_\Sigma\times \Sigma \rightarrow \M_\Sigma$.

The family of differential operators $\DiracS_\bullet$ is equivariant for the full $\G_\Sigma$-action, but the boundary condition is only equivariant for the action of a subgroup.
\begin{proposition}
\label{p:gaugetransfamily}
Under the diagonal action of $g \in \G_\Sigma$ on $\A_\Sigma^\flat\times L^2(\Sigma,S^\Sigma\boxtimes V)$, the family $D_\bullet=(\DiracS_\bullet,B_{<0}(\partial))$ is sent to the family
\[ (\DiracS_\bullet,B_{< 0}(\partial_{\scr{R}(g)\cdot 0})), \]
where $\scr{R}(g)$ is the restriction of $g$ to the boundary. Consequently $D_\bullet$ (or equivalently the bounded transform $F_\bullet$) is equivariant for the subgroup $\G_{\Sigma,\partial \Sigma-lc}$ consisting of $g \in \G_\Sigma$ such that $\scr{R}(g)$ is locally constant on $\partial \Sigma$; equivalently $\scr{R}(g) \in L\sul{G}_0=\sul{G}$, the stabilizer of $0 \in L\ul{\g}^*$ under the gauge action.
\end{proposition}
\begin{proof}
Under the action of $g \in \G_\Sigma$ on $\E_0$, the operator $(\DiracSP_A,B_{<0}(\partial))$ is sent to the operator $(\DiracSP_{g\cdot A},B_{< 0}(\partial_{\scr{R}(g)\cdot 0}))$, where the action of $g$ on $A$ (resp. $\scr{R}(g)$ on $0$) is by gauge transformations.  If $g \in \G_{\Sigma,\partial \Sigma-lc}$ then $\scr{R}(g)\cdot 0=-\d \scr{R}(g) \scr{R}(g)^{-1}=0$ by definition, hence $g$ maps the family into itself.
\end{proof}

Since $\G_{\Sigma,\partial \Sigma} \subset \G_{\Sigma,\partial \Sigma-lc}$, passing to $\G_{\Sigma,\partial \Sigma}$-equivalence classes yields a $\G_{\Sigma,\partial \Sigma-lc}/\G_{\Sigma,\partial \Sigma}\simeq \sul{G}$-equivariant family of Fredholm operators over $\M_\Sigma$. Recall $q \colon \A_\Sigma^\flat\rightarrow \A_\Sigma^\flat/\G_{\Sigma,\partial \Sigma}=\M_\Sigma$ is the quotient map.
\begin{definition}
\label{d:opfam}
Let $D_{q(\bullet)}$ denote the $L\sul{G}_0=\sul{G}$-equivariant family of unbounded Fredholm operators acting on the fibers of $\E$ and parametrized by $\M_\Sigma$, which is obtained by descending the family $D_\bullet$ along the fibers of $q$. Likewise let $F_{q(\bullet)}$ denote the family of bounded transforms.
\end{definition}

The family $F_{q(\bullet)}$ varies norm-continuously. Indeed the family $F_\bullet$ on $\A_\Sigma^\flat$ is norm-continuous, on account of the fact that the boundary condition $B_{<0}(\partial)$ is fixed, and the differential operators $\DiracS_\bullet$ all have the same symbol. As explained in Proposition \ref{p:charts}, $\A_\Sigma^\flat\rightarrow \M_\Sigma$ is a principal $\G_{\Sigma,\partial \Sigma}$-bundle. Elements of $\G_{\Sigma,\partial \Sigma}$ are of Sobolev class $\varsigma+\frac{1}{2}>2$, and as the latter Sobolev norm dominates the supremum norm, the map $\G_{\Sigma,\partial \Sigma}\rightarrow \scr{U}(\E_0)$ is continuous with respect to the norm topology. Thus the structure group of the Hilbert bundle $\E$ can be taken to be the unitary group $\scr{U}(\E_0)$ with the norm topology, and continuity of $F_{q(\bullet)}$ follows.

\begin{definition}
\label{d:Esigma}
The pair $(\E,F_{q(\bullet)})$ defines a class $\sf{E}^\Sigma V \in K^0_{\sul{G}}(\M_\Sigma)$.
\end{definition}

\subsection{Affine Weyl group symmetry.}\label{s:affineWeylaction}
Let $\Gamma$ be a topological group and let $\Gamma' \subset \Gamma$ a compact normal subgroup. If $Y$ is a $\Gamma$-space, then the component group $\pi_0(\Gamma/\Gamma')$ acts on $K^*_{\Gamma'}(Y)$. When $\gamma \in \Gamma$ commutes with $\Gamma'$, the action is simply pullback by $\gamma^{-1}$. In particular for $Y=\M_\Sigma$ there is an action of the affine Weyl group $\sul{W}_{\aff}=W_{\aff}^b$ on $K^*_{\sul{T}}(\M_\Sigma)$. Forgetting the $\sul{G}$-action, the Atiyah-Bott classes $\sf{E}^NV$ restrict to classes in $K^*_{\sul{T}}(\M_\Sigma)$ that we denote by the same symbols. For $\dim(N)=0,1$ the resulting class is $\sul{W}_{\aff}$-invariant.

Let $\Tr_V \colon T \rightarrow \bC$ be the restriction of the character of the finite dimensional representation $V$ to the maximal torus $T \subset G$. It admits a decomposition
\[ \Tr_V=\sum_{\lambda \in \Lambda} \dim(V_{[\lambda]})e^\lambda \]
where $V_{[\lambda]}=\Hom(\bC_\lambda,V)$ is the $\lambda$ weight space.  Let $\ul{\Tr}_V$ be the character of $V^{\oplus b}$ viewed as a representation of $\sul{T}=T^b$,
\[ \ul{\Tr}_V \colon \ul{u}=(u_1,...,u_b)\in \sul{T}=T^b\mapsto \Tr_V(u_1)+\cdots+\Tr_V(u_b)\in \bC.\]
If $\eta \in \Pi$ let
\begin{equation}
\label{e:grad}
\partial_\eta \Tr_V=\sum_{\lambda \in \Lambda} \dim(V_{[\lambda]})2\pi \i\pair{\lambda}{\eta}e^\lambda,
\end{equation} 
be the derivative of $\Tr_V$ in the direction $\eta$. Since $\pair{\lambda}{\eta} \in \bZ$, $(2\pi \i)^{-1}\partial_\eta \Tr_V \in R(T)$. Similarly for $\ul{\eta}=(\eta_1,...,\eta_b)\in \ul{\Pi}$ and $\ul{u}=(u_1,...,u_b)$, let
\[ (\partial_{\ul{\eta}}\ul{\Tr}_V)(\ul{u})=\sum_{j=1}^b(\partial_{\eta_j}\Tr_V)(u_j).\]
\begin{proposition}
\label{p:affineWeylActionKthy}
The Weyl group $\sul{W}=W^b$ acts on $\sf{E}^\Sigma V \in K^0_{\sul{T}}(\M_\Sigma)$ trivially, while $\ul{\eta} \in \ul{\Pi}$ acts as $\ul{\eta}^*\sf{E}^\Sigma V=\sf{E}^\Sigma V-(2\pi \i)^{-1}\partial_{\ul{\eta}} \ul{\Tr}_V$.
\end{proposition}
\begin{proof}
The class $\sf{E}^\Sigma V$ comes from a class in $K^0_{\sul{G}}(\M_\Sigma)$, hence $\sul{W}$ acts trivially. Recall $\ul{\eta}=(\eta_1,...,\eta_b) \in \ul{\Pi}$ embeds in $L\sul{G}$ as the exponential loop $s\mapsto (\exp(-s\eta_1),...,\exp(-s\eta_b))$, $s \in [0,1]$. The corresponding gauge transformation along $\partial_j \Sigma$ is $s_j\mapsto \exp(-s_j\eta_j)$. Using Proposition \ref{p:gaugetransfamily} (for $\ul{\eta}^{-1}$) $\ul{\eta}^*\sf{E}^\Sigma V$ is represented by the family of boundary problems obtained by descending the family $(\DiracSP_A,B_{< 0}(\partial_{-\ul{\eta}}))_{A \in \A_\Sigma^\flat}$. From equation \eqref{e:partiala} and since $\c^{\partial \Sigma}(\partial_s)$ is multiplication by $\i$, the operator $\partial_{\ul{\eta}}$ is
\begin{equation} 
\label{e:bdop}
\partial - \i\pi_V(\ul{\eta}), 
\end{equation}
where we view $\ul{\eta}$ as the locally constant function on the boundary equal to $\eta_j$ on the $j$-th boundary component. We will deduce the result from \eqref{e:bdop} and a families version of the formula for the change of index under a change of boundary conditions (cf. \cite[Corollary 8.8, Theorem 8.15]{BarBallmann} for the standard version).

Let $P_1$ (resp. $P_2$,$P_0$) denote $L^2$-orthogonal projection onto $B_1=B_{\ge 0}(\partial)$ (resp. $B_2=B_{\ge 0}(\partial_{\ul{\eta}})$, $B_0=B_1+B_2$).  For $i=0,1,2$ consider the families $\bm{D}_i$ of operators over $\A_\Sigma^\flat$:
\[ (\DiracSP_A,P_i\circ \scr{R})\colon L^{2,1}(\Sigma, S^{\Sigma,+}\boxtimes V) \rightarrow L^2(\Sigma, S^{\Sigma,-}\boxtimes V)\oplus B_i.\]
Upon descending to $\M_\Sigma$, the families $\bm{D}_1$, $\bm{D}_2$ represent $\sf{E}^\Sigma V$, $\ul{\eta} \cdot \sf{E}^\Sigma V$ respectively. 

By construction
\[ \bm{D}_i=(\id,p_i)\circ \bm{D}_0 \]
holds for $i=1,2$, where $p_i=P_i|_{B_0}$ and $\circ$ denotes (pointwise) composition of families of operators over $\A_\Sigma^\flat$. By additivity of the index for compositions (families version) we obtain the following equations in $\R\sK^0_{\sul{T}}(\M_\Sigma)$,
\[ \sf{E}^\Sigma V=[\bm{D}_0]+[p_1], \qquad \ul{\eta} \cdot \sf{E}^\Sigma V=[\bm{D}_0]+[p_2] \]
hence
\[ \ul{\eta} \cdot \sf{E}^\Sigma V=(\ul{\eta}^{-1})^*\sf{E}^\Sigma V=\sf{E}^\Sigma V+[p_2]-[p_1].\]
We have $[p_1]=[B_1^\perp\cap B_2]$, $[p_2]=[B_2^\perp\cap B_1]$, and these subspaces are easy to determine. Consider $B_2^\perp\cap B_1=B_{<0}(\partial_{\ul{\eta}})\cap B_{\ge 0}(\partial)$. Let $\lambda \in \Lambda$. Restricted to sections of $V_{[\lambda]}$ over the $j$-th component of $\partial \Sigma$, equation \eqref{e:bdop} simplifies to $\partial-2\pi\pair{\lambda}{\eta_j}$. Recalling that the spectrum of $\partial$ is $(2\bZ+1)\pi$, we deduce that for the $\lambda$-isotypical component 
\[ \dim\big(B_{<0}(\partial-2\pi\pair{\lambda}{\eta_j})\cap B_{\ge 0}(\partial)\big)_{[\lambda]}=\begin{cases} \pair{\lambda}{\eta_j}\dim(V_{[\lambda]}), & \text{ if } \pair{\lambda}{\eta_j}>0\\ 0, & \text{ else.} \end{cases} \]
One finds a similar result for $B_1^\perp\cap B_2$, except with $\pair{\lambda}{\eta_j}$ replaced everywhere with $-\pair{\lambda}{\eta_j}$. Summing over $\lambda$, the contribution of the $j$-th boundary component to the difference $[p_2]-[p_1]$ is
\[ \sum_\lambda \pair{\lambda}{\eta_j}\dim(V_{[\lambda]})e_\lambda=(2\pi \i)^{-1}\partial_{\eta_j}\Tr_V.\]
Summing over $j$ gives the result.
\end{proof}

\begin{remark}
The action of $W_\aff$ cannot change the virtual dimension of the K-theory class. Hence the virtual dimension of $\partial_{\ul{\eta}} \ul{\Tr}_V$ should be $0$ for consistency. The virtual dimension is the sum over $j$ of $\partial_{\eta_j}\Tr_V$ evaluated at $1 \in G$, which is $0$ by $G$-invariance of the gradient of $\Tr_V$.
\end{remark}
\begin{remark}
If $m \in \M_\Sigma$ is a $\sul{T}$-fixed point, then 
\[ \sf{E}^\Sigma V|_m=\index_{\sul{T}}(D_m)\in K^0_{\sul{T}}(\pt)=R(\sul{T}). \] 
Proposition \ref{p:affineWeylActionKthy} shows that 
\[ \sf{E}^\Sigma V|_{\ul{\eta}\cdot m}=\sf{E}^\Sigma V|_m-(2\pi \i)^{-1}\partial_{\ul{\eta}} \ul{\Tr}_V \in R(\sul{T}).\] 
In particular this implies that the dimension of the kernel of the operators in the family $D_\bullet$ is unbounded. It also rules out the possibility that the class $\sf{E}^\Sigma V|_{\M_\Sigma^{\sul{T}}} \in K^0_{\sul{T}}(\M_\Sigma^{\sul{T}})$ is pulled back from any $\sul{T}$-equivariant compactification of $\M_\Sigma^{\sul{T}}$, because if that were so then the $\sul{T}$-equivariant index of $D_m$ would be bounded as $m$ varies in $\M_\Sigma^{\sul{T}}$.
\end{remark}

\section{The index of the Atiyah-Bott classes}
In \cite{LIndexFormula} we constructed (in the more general context of an arbitrary proper Hamiltonian loop group space) a sequence of $R(\sul{T})$-module homomorphisms
\begin{equation} 
\label{e:indexhom}
\index^k_{\sul{T}}\colon K^\ad_{\sul{T}}(\M_\Sigma)\rightarrow R^{-\infty}(\sul{T})
\end{equation}
out of an $R(\sul{T})$-subalgebra of $K_{\sul{T}}(\M_\Sigma)=K^0_{\sul{T}}(\M_\Sigma)$ consisting of K-classes satisfying an equivariant bounded geometry condition that we referred to as admissible classes. For classes in the image of the forgetful map $K^\ad_{\sul{G}}(\M_\Sigma)\rightarrow K^\ad_{\sul{T}}(\M_\Sigma)$, the index takes values in $R^{-\infty}(\sul{T})^{W-\tn{anti}}$, the $\sul{W}$-antisymmetric formal Fourier series, i.e. formal series
\[ \sum_{\ul{\lambda} \in \ul{\Lambda}} m(\ul{\lambda})e^{\ul{\lambda}}, \quad \text{such that} \quad m(\ul{w}\cdot \ul{\lambda})=(-1)^{l(\ul{w})}m(\ul{\lambda}), \quad \forall \ul{w} \in \sul{W}, \ul{\lambda} \in \ul{\Lambda}.\]
There is an induction isomorphism $R^{-\infty}(\sul{T})^{W-\tn{anti}}\simeq R^{-\infty}(\sul{G})$, hence there are similar $R(\sul{G})$-module homomorphisms
\[ \index^k_{\sul{G}}\colon K^\ad_{\sul{G}}(\M_\Sigma)\rightarrow R^{-\infty}(\sul{G}).\]
In this section we recall the definition of the index homomorphism \eqref{e:indexhom} and prove that the Atiyah-Bott classes are admissible.

\subsection{Global transversals}\label{s:globtrans}
The orbits of the action of $G$ on itself by conjugation are transverse to any small tubular neighborhood of $T$ inside $G$. Such a neighborhood can be lifted to a `thickening' of the subspace $\t \subset L\g^*=\Omega^1(S^1,\g)$ of constant $\t$-valued 1-forms. Let $\t^\perp$ be the orthogonal complement to $\t$ inside $\g$. Let $N(T)\ltimes \Pi$ act on $\t$ via the standard action of the affine Weyl group $W_\aff=W\ltimes \Pi=(N(T)\ltimes \Pi)/T$, and on $\t^\perp$ via the adjoint action of $N(T)$ (with $\Pi$ acting trivially). Let $\ol{B}_\epsilon(\t^\perp)$ denote the closed ball in $\t^\perp$ of radius $\epsilon$.
\begin{proposition}[\cite{LMSspinor}, Section 6.4]
\label{p:trans}
For $\epsilon>0$ sufficiently small, the inclusion $\t \hookrightarrow L\g^*$ extends to an $N(T)\ltimes \Pi$-equivariant embedding
\[ \t \times \ol{B}_\epsilon(\t^\perp) \rightarrow L\g^* \]
with image $R \subset L\g^*$ transverse to the $\Omega G$ orbits and contained in the dense subspace of smooth loops. Each $\Omega G$ orbit intersects $R$ non-trivially in an orbit of $\Pi$.
\end{proposition}
The construction is canonical up to homotopy, the only choices involved being a splitting (preserved under convex combinations) and a `thickness' $\epsilon$. For further discussion see also \cite{LIndexFormula}.

\begin{definition}[\cite{LMSspinor}]
\label{d:globtrans}
Let $R\subset L\g^*$ be as in Proposition \ref{p:trans}, and let $\ul{R}\subset L\ul{\g}^*$ be the product of $b$ copies of $R$. As $\mu_{\M_\Sigma}$ is $L\sul{G}$-equivariant, the inverse image
\[ X=\mu_{\M_\Sigma}^{-1}(\ul{R}) \]
is a finite dimensional $N(\sul{T})\ltimes \ul{\Pi}$-invariant submanifold (with boundary) of $\M_\Sigma$ that we will refer to as a \emph{global transversal} of $\M_\Sigma$. The interior of $X$ (resp. $R$) is denoted $X^\circ$ (resp. $R^\circ$). The moment map $\mu_{\M_\Sigma}$ restricts to a proper $N(\sul{T})\ltimes \ul{\Pi}$-equivariant map $\mu_X \colon X \rightarrow \ul{R}$.
\end{definition}

\subsection{Spinor bundles and Dirac operators}
A vector bundle $E\rightarrow \M_\Sigma$ is said to be at \emph{level} $\ell \in \bZ$ if it carries an action of $\wh{L\sul{G}}^{(\ell)}$ compatible with the action of $L\sul{G}$ on $\M_\Sigma$. Likewise we say that a vector bundle $E\rightarrow X$ is at \emph{level} $\ell$ if it is the restriction to $X$ of a vector bundle at level $\ell$ on $\M_\Sigma$. $\M_\Sigma$ has a canonical level $1$ prequantum line bundle
\begin{equation}
\label{e:prequant}
L\rightarrow \M_\Sigma,
\end{equation}
obtained by reduction of a prequantum line bundle for the Atiyah-Bott symplectic structure on $\A_\Sigma$. For details see \cite[Section 3.3]{MWVerlindeFactorization}, following \cite{Mickelsson, RamadasSingerWeitsman, Witten}.

A K-theory orientation for an even-dimensional manifold $Y$ is a Spin$_c$ structure, or equivalently a spinor bundle, i.e. a complex vector bundle $\scr{S}\rightarrow Y$ whose fibers are irreducible modules for the bundle of Clifford algebras $\bC l(TY,g_Y)$. The manifold $X$ has a canonical K-theory orientation.
\begin{theorem}[\cite{LMSspinor}]
\label{t:spinor}
There is a canonical spinor bundle $\scr{S}_{\tn{can}} \rightarrow X$, which is the restriction of an $\wh{L\sul{G}}^{(\hvee)}$-equivariant vector bundle on $\M_\Sigma$, where $\hvee$ is the dual Coxeter number of $G$.
\end{theorem}
The dual Coxeter number appears because it is the level of the spin representation of the loop group. Further spinor bundles are obtained by tensoring with $L$. Thus for any $k>-\hvee$ one has a canonical spinor bundle
\begin{equation} 
\label{e:levellspinor}
\scr{S}=\scr{S}_k=\scr{S}_{\tn{can}}\otimes L^k, \qquad \text{at level } \ell=k+\hvee>0.
\end{equation}
The corresponding \emph{anti-canonical line bundle} is the level $2\ell$ line bundle
\begin{equation}
\label{e:anticanon}
\scr{L}=\scr{L}_k=\Hom_{\bC l(TX)}(\scr{S}^*,\scr{S})=\scr{L}_{\tn{can}}\otimes L^{2k}.
\end{equation}

Choosing a complete invariant Riemannian metric $g$ on the interior $X^\circ$ and an invariant connection on $\scr{L}$ determines a connection on $\scr{S}$ compatible with the Levi-Civita connection. The corresponding Spin$_c$ Dirac operator $\Dirac_{\scr{S}}$ is defined as the composition
\[ C_c^\infty(X^\circ,\scr{S})\xrightarrow{\nabla}C_c^\infty(X^\circ,T^*X^\circ\otimes \scr{S})\xrightarrow{\c}C_c^\infty(X^\circ,\scr{S}).\]
Recall that $\ul{R}$ was a `thickening' of $\ul{\t}\subset L\ul{\g}^*$, and likewise we view $X^\circ$ is a `thickening' of the singular subset $\mu_{\M_\Sigma}^{-1}(\ul{\t})$; to account for this we introduce a new operator $\Dirac$ obtained by twisting $\Dirac_{\scr{S}}$ by the pullback under the composition $X^\circ \xrightarrow{\mu_{\M_\Sigma}} \ul{R}\rightarrow \ul{\t}^\perp$ of a cycle representing the K-theoretic Thom-Bott class 
\begin{equation}
\label{e:ThomBott}
\scr{B} \in K^0_{\sul{T}}(\ul{\t}^\perp).
\end{equation} 
The resulting operator can be thought of as representing the virtual fundamental class of $\mu_{\M_\Sigma}^{-1}(\ul{\t})$; it is a Dirac-type operator
\begin{equation} 
\label{e:Diractype}
\Dirac \colon C_c^\infty(X^\circ,S)\rightarrow C^\infty_c(X^\circ,S) 
\end{equation}
acting on sections of the $\bZ_2$-graded vector bundle
\begin{equation}
\label{e:Diractypespinor}
S=\scr{S}_{\ul{\t}^\perp}\wh{\boxtimes} \scr{S}, \qquad \scr{S}_{\ul{\t}^\perp}=\wedge \ul{\n}^- \otimes \bC_{\ul{\rho}},
\end{equation}
where $\scr{S}_{\ul{\t}^\perp}$ is the $\sul{T}$-equivariant $\bZ_2$-graded spin representation for $\Cl(\ul{\t}^\perp)$, $\n^-\subset \g_\bC$ is the direct sum of the negative root spaces and $\rho$ is the half sum of the positive roots. By Dirac-type, we mean a Dirac operator up to a $0$-th order term, and indeed $\Dirac$ incorporates a $0$-th order potential term whose square goes to infinity at $\partial X$. See \cite{LIndexFormula} for details.

\subsection{The index homomorphism}
If $Y$ is a compact even-dimensional Spin$_c$ manifold and $f\colon Y\rightarrow M$ is a smooth map, then there is a homomorphism
\[ \index \colon K(M)=K^0(M)\rightarrow \bZ \]
taking a K-theory class represented by a vector bundle $E$ to the index of the Fredholm operator obtained by coupling the Spin$_c$ Dirac operator on $Y$ to $f^*E$ with any connection. In case a compact Lie group acts compatible with all the data, there is a similar equivariant index homomorphism to the representation ring of the compact Lie group. 

The index homomorphism 
\begin{equation} 
\label{e:indexhom2}
\index^k_{\sul{T}}\colon K^\ad_{\sul{T}}(\M_\Sigma)\rightarrow R^{-\infty}(\sul{T})
\end{equation}
is analogous, with an important difference: it is defined using the \emph{non-compact} submanifold $X\hookrightarrow \M_\Sigma$ and Dirac-type operator $\Dirac$ built using the level $\ell=k+\hvee>0$ Spin$_c$ structure $\scr{S}_{\tn{can}}\otimes L^k$. Non-compactness of $X$ leads to several complications:
\begin{enumerate}
\item The homomorphism is only defined on an $R(\sul{T})$-subalgebra $K^{\ad}_{\sul{T}}(\M_\Sigma)\subset K_{\sul{T}}(\M_\Sigma)$ of admissible classes. These are K-classes represented by Fredholm complexes satisfying an equivariant bounded geometry condition guaranteeing that the index is well-defined. Moreover $[L] \notin K^{\ad}_{\sul{T}}(\M_\Sigma)$, which is why we have a sequence $\index^k_{\sul{T}}$ (with $k>-\hvee$) rather than a single index homomorphism.
\item The index takes values in formal Fourier series. This is because, even under the admissibility assumption, the $L^2$ kernel and cokernel of the operator will be infinite dimensional in general, although the multiplicity of each irreducible representation of $\sul{T}$ will be finite.
\item On a non-compact space $X$, it is not true that every K-theory class can be represented as a formal difference of vector bundles. This is also false for the more restricted collection of admissible classes. Therefore a more general construction is necessary, in which Fredholm complexes are coupled to $\Dirac$.
\item Rather than being essentially automatic as it is in the compact case, it requires an argument (based on the non-abelian localization formula) that the resulting Fredholm index only depends on the K-theory class, and so descends to a map \eqref{e:indexhom2}.
\end{enumerate}
The index homomorphism was constructed along these lines for general Hamiltonian loop group spaces in \cite{LIndexFormula}. To apply the general theory to the Atiyah-Bott classes, it remains to prove that they are admissible. The definition of admissibility in the special case of the classes $\sf{E}^\Sigma V$ will be reviewed in the course of the proof of their admissibility below.

Any class in $K_{\sul{G}}(\M_\Sigma)$ that is the pullback of a $\sul{G}$-equivariant class on $M_\Sigma=\M_\Sigma/\Omega \sul{G}$ is admissible. The Atiyah-Bott classes constructed from $N\subset \Sigma$ with $\dim(N)=0,1$ are such pullbacks, since the data defining these classes was in fact $L\sul{G}$-equivariant. Thus for these classes admissibility is immediate.

We shall thus focus on the class $\sf{E}^\Sigma V$ defined in Section \ref{s:ABSigma}. Recall this is the class represented by the pair $(\E,D)$, where $\E \rightarrow \M_\Sigma$ is a Hilbert bundle with fibers $L^2(\Sigma,S^\Sigma \boxtimes V)$ and $D=D_{q(\bullet)}$ is a family of elliptic boundary problems.

The Hilbert bundle $\E|_X \rightarrow X$ is $N(\sul{T})\ltimes \ul{\Pi}$-equivariant. To check admissibility, it will be useful to have a somewhat explicit description of an invariant Hermitian connection $\nabla^\E$ on $\E|_X$ built using Coulomb gauge charts (see Section \ref{s:modsp}). Choose a compact fundamental domain $X_\diamond \subset X$ for the $\ul{\Pi}$ action, and choose a finite number $i=1,...,N$ of Coulomb gauge charts whose restrictions $U_{i,\diamond}$ to $X$ cover $X_\diamond$. By making the charts smaller if necessary, we may assume that $\ul{\eta} \cdot U_{i,\diamond}$, $\ul{\eta}'\cdot U_{i,\diamond}$ are disjoint when $\ul{\eta} \ne \ul{\eta}'$, and moreover that the closures $\ol{U}_{i,\diamond}$ are still contained in the domain of some slightly larger Coulomb chart. Let
\[ U_i=\bigcup_{\ul{\eta} \in \ul{\Pi}} \ul{\eta} \cdot U_{i,\diamond}\simeq \ul{\Pi}\times U_{i,\diamond}.\]
Thus $\U=\{U_1,...,U_N\}$ is a finite open cover of $X$ by $\ul{\Pi}$-invariant subsets. The covering is obviously uniformly finite, in the sense that there is a $\delta>0$ such that for every $x \in X$, there is an $i \in \{1,...,N\}$ such that the ball of radius $\delta$ around $x$ is contained in $U_i$. Let us also fix a partition of unity $\{\rho_1,...,\rho_N\}$ subordinate to the cover, with each $\rho_i$ being $\ul{\Pi}$-invariant. 

The choice of a Coulomb gauge chart $U_{i,\diamond}$ determines a trivialization $\Phi_{i,\diamond}\colon \E\upharpoonright U_{i,\diamond}\rightarrow \E_0\times U_{i,\diamond}$, and hence a Hermitian connection $\nabla^{\E,U_{i,\diamond}}$. Since the bundle $\E|_X$ is $\ul{\Pi}$-equivariant, we may use the $\ul{\Pi}$ action to obtain a connection $\nabla^{\E,U_i}$ defined over all of $U_i=\ul{\Pi}\cdot U_{i,\diamond}$, and we define
\[ \nabla^{\E,\U}=\sum_{i=1}^m \rho_i \nabla^{\E,U_i}.\]
Equivalently the trivializations and partition of unity determine a $\ul{\Pi}$-equivariant embedding of $\E|_X$ into the trivial bundle $X \times (\E_0 \otimes \bC^N)$, and $\nabla^{\E,\U}$ is the connection associated to this embedding. To obtain an $N(\sul{T})\ltimes \ul{\Pi}$-invariant connection $\nabla^\E$, average $\nabla^{\E,\U}$ with the action of the compact group $N(\sul{T})$.

The Riemannian metric $g_X$ on $X$, the Hermitian structure on $\E$, and the connection $\nabla^{\E,\U}$ determine a Riemannian metric on the unit disk bundle $\B \subset \E|_X$. The torus $\sul{T}$ acts on $\B$. Let $\ul{\xi}_\B$ be the vector field on $\B$ induced by $\ul{\xi} \in \ul{\t}$.
\begin{lemma}
\label{l:ETfinite}
The vector field $\ul{\xi}_\B$ has bounded norm.
\end{lemma}
\begin{proof}
The horizontal component (relative to $\nabla^{\E,\U}$) of $\ul{\xi}_\B$ is the horizontal lift of $\ul{\xi}_X$, which clearly has bounded norm (it is $\ul{\Pi}$-invariant). By $\ul{\Pi}$-invariance of the norms and connections, together with the fact that the $\sul{T}$, $\ul{\Pi}$ actions commute, it is enough to check that the $\nabla^{\E,U_{i,\diamond}}$-vertical component of $\ul{\xi}_\B$ has bounded norm over each of the compact closures $\ol{U}_{i,\diamond} \subset X$, $i=1,...,N$ (recall we assumed that the closures $\ol{U}_{i,\diamond}$ is contained in some slightly larger Coulomb chart, so $\nabla^{\E,U_{i,\diamond}}$ extends over this slightly larger subset).

The unit ball bundle $\B=(\A_\Sigma^\flat\times \B_0)/\G_{\Sigma,\partial \Sigma}$, where $\B_0\subset \E_0$ is the unit ball. The vector field $\ul{\xi}_\B$ can be described as follows: let $\bm{\xi}$ be any smooth $\g$-valued function on $\Sigma$ whose restriction to the $j$-th boundary component $\partial_j\Sigma$ of $\Sigma$ is $\xi^j \in \t$. Then $\bm{\xi}$ induces a vector field $(\bm{\xi}_{\A_\Sigma^\flat},\bm{\xi}_{\B_0})$ on $\A_\Sigma^\flat\times \B_0$ which descends to $\ul{\xi}_\B$. The Coulomb gauge condition uniquely determines a continuous (because $\varsigma+\frac{1}{2}>1$) map $\bm{\zeta} \colon \ol{U}_{i,\diamond}\times \Sigma \rightarrow \g$ such that the (partially-defined) vector field $\bm{\xi}'_{\A_\Sigma^\flat}$ is tangent to the Coulomb slice, where $\bm{\xi}'\colon \ol{U}_{i,\diamond}\times \Sigma\rightarrow \g$ is the map $\bm{\xi}'(x,p)=\bm{\xi}(p)-\bm{\zeta}(x,p)$, for $x \in \ol{U}_{i,\diamond}$, $p \in \Sigma$. The norm of the $\nabla^{\E,U_{i,\diamond}}$-vertical component of $\ul{\xi}_\B$ over $\B\upharpoonright \ol{U}_{i,\diamond}$ is bounded by
\[ \sup_{(x,p)}|\pi_V(\bm{\xi}'(x,p))|_{\End(V)},\]
where the supremum is taken over $(x,p)\in \ol{U}_{i,\diamond}\times \Sigma$. This is finite because $\bm{\xi}' \colon \ol{U}_{i,\diamond}\times \Sigma \rightarrow \g$ is bounded and $(V,\pi_V)$ is a finite dimensional representation.
\end{proof}

\begin{theorem}
\label{t:quasiperiodicity}
The Atiyah-Bott class $\sf{E}^\Sigma V$ is admissible.
\end{theorem}
\begin{proof}
We will prove that $(\E,D)$ is an admissible cycle for $\sf{E}^\Sigma V$ in the sense of \cite{LIndexFormula}. There are four properties to check. The first property is that $(1+D_x^2)^{-1}$ be compact for each $x \in X$, which holds because $D_x$ is an elliptic boundary value problem on the compact manifold $\Sigma$. The second property is Lemma \ref{l:ETfinite}. The third property is a very mild smoothness condition on the family $D_x, x \in X$ (we refer the reader to \cite{LIndexFormula} for the precise definition) and is immediate in this case. The last property is that $\nabla^\E$ preserves the domain subbundle $\dom(D_{q(\bullet)}) \subset \E$, and that the covariant derivative $\nabla^\E D$ is a smooth bounded section of $T^*X\otimes \scr{B}(\E)$, where $\scr{B}(\E)$ denotes the bundle of bounded endomorphisms of the fibers.

Each of the Coulomb gauge connections $\nabla^{\E,U_{i,\diamond}}$ preserves the domain of $D_{q(\bullet)}$ since the boundary condition is constant. The extension of $\nabla^{\E,U_{i,\diamond}}$ to a connection $\nabla^{\E,U_i}$ over $U_i=\ul{\Pi}\cdot U_{i,\diamond}$ using the $\ul{\Pi}$ action also has this property, because the image of a Coulomb chart under the action of an element of the gauge group is again a Coulomb chart. Since $\nabla^\E$ is built from the $\nabla^{\E,U_i}$ by patching together with a partition of unity and then averaging over $N(\sul{T})$, $\nabla^\E$ parallel translation also preserves $\dom(D_{q(\bullet)})$.

Next we argue that the $\nabla^\E$-covariant derivative of $D$ is a bounded section of $T^*X\otimes \scr{B}(\E)$. By the construction of $\nabla^\E$, it is enough to prove instead that the $\nabla^{\E,U_i}$-covariant derivative of $D|_{U_i}$ is a bounded section of $T^*U_i\otimes \scr{B}(\E|_{U_i})$, for each $i=1,...,N$ ($\rho_1,...,\rho_N$ are periodic so have globally bounded gradient). Fix $i$ and let $U_\diamond=U_{i,\diamond}$, $U=U_i\simeq \ul{\Pi} \times U_\diamond$. Recall that $U_\diamond \subset X$ is the restriction of a Coulomb gauge chart for some connection $A_\diamond$ on $\Sigma$. Let $\psi \colon U_{\diamond} \rightarrow \A_\Sigma^\flat$ be the smooth local section of the bundle $\A_\Sigma^\flat\rightarrow \M$ over $U_{\diamond}$ that is determined by the chart. By smoothness of $\psi$, and since $U_{\diamond}$ is relatively compact, there is an estimate of the form
\[ \|B\|_{\varsigma-\frac{1}{2}}\le C|Tq(B)| \]
for $B \in T\psi(U_{\diamond})$, where $Tq$ is the tangent map and $|Tq(B)|$ is computed using the Riemannian metric $g_X$. Since $\varsigma>\frac{3}{2}$, Sobolev embedding yields an estimate
\begin{equation} 
\label{e:supnormest}
\|B\|_{\infty} \le C|Tq(B)|,
\end{equation}
for $B \in T\psi(U_{\diamond})$.

For each $\ul{\eta} \in \ul{\Pi}$, let $D_{\ul{\eta},q(\bullet)}$ denote the family of operators $\ul{\eta}^{-1}\circ (D_{q(\bullet)}\upharpoonright \ul{\eta}\cdot U_{\diamond}) \circ \ul{\eta}$ on $\E\upharpoonright U_{\diamond}$ obtained by restricting the family $D_{q(\bullet)}$ to $\ul{\eta} \cdot U_{\diamond}$ and pulling back to $U_{\diamond}$ using the $\ul{\Pi}$-action on $\E$. By construction of $\nabla^{\E,U}$, boundedness of the derivative of $D_{q(\bullet)}$ over $U$ amounts to showing that $\nabla^{\E,U_{\diamond}}D_{\ul{\eta},q(\bullet)}$ is a bounded operator for each $\ul{\eta}$, and moreover the norm is bounded by a constant that is independent of $\ul{\eta}$.

By Proposition \ref{p:gaugetransfamily}, the family $D_{\ul{\eta},q(\bullet)}$ over $U_{\diamond}$ is obtained by descending a family $D_{\ul{\eta},\bullet}$ along the fibers of $q \colon \A_\Sigma^\flat\rightarrow \M_\Sigma$ (over $U_{\diamond}$), where $D_{\ul{\eta},A}$ is the odd self-adjoint operator constructed from the family $(\DiracSP_A,B_{<0}(\partial_{\ul{\eta}}))$ as in Definition \ref{d:thefamily}. Fix $A \in \psi(U_{\diamond})$, and let $B \in T_A \psi(U_{\diamond})$. Then by definition
\begin{equation} 
\label{e:derivative1}
\nabla^{\E,U_\diamond}_{Tq(B)}D_{\ul{\eta},q(\bullet)}=\lim_{t\rightarrow 0}t^{-1}(D_{\ul{\eta},A+tB}-D_{\ul{\eta},A})=(\c^\Sigma\boxtimes \pi_V)(B),
\end{equation}
a bounded operator, with operator norm at most $\|B\|_\infty$ independent of $\ul{\eta}$. We conclude that the derivative satisfies
\[ \|\nabla^{\E,U_\diamond}_{Tq(B)}D_{\ul{\eta},q(\bullet)}\|\le \|B\|_\infty\le C|Tq(B)|, \]
where the second inequality is \eqref{e:supnormest}. This verifies the required boundedness property. The smoothness of the $\nabla^\E$-covariant derivative is clear from \eqref{e:derivative1}, and this completes the proof.
\end{proof}

The connection $\nabla^\E$ is used to extend $\Dirac$ to an operator $\Dirac_{\E}$ acting on sections of $S \wh{\otimes} \E|_X$. The operator obtained by \emph{coupling $\Dirac$ to $(\E,D)$ with $\nabla^\E$} is then by definition
\begin{equation} 
\label{e:coupled}
\Dirac_{\E}+1\wh{\otimes} D|_X,
\end{equation}
with initial domain consisting of smooth compactly supported sections of $S\wh{\otimes}\dom(D|_X)$. Admissibility (Theorem \ref{t:quasiperiodicity}) and the results of \cite{LIndexFormula} imply that the $\sul{T}$-equivariant $L^2$ index of \eqref{e:coupled} is well-defined in $R^{-\infty}(\sul{T})$, and this is the definition of the index homomorphism for $\sf{E}^\Sigma V$. As mentioned above, it is a consequence of the non-abelian localization formula proved in \cite{LIndexFormula} that this index only depends on the K-theory class $\sf{E}^\Sigma V$.

By Proposition \ref{p:affineWeylActionKthy}, the class $\sf{E}^\Sigma V$ is unchanged under the natural action of the Weyl group $\sul{W}$ on $K_{\sul{T}}(\M_\Sigma)$. The same is true of the other Atiyah-Bott classes. On the other hand, the Bott-Thom class built into the definition of $\Dirac$ is antisymmetric under the action of the Weyl group. As a result the index of an Atiyah-Bott class is antisymmetric under the action of the Weyl group \cite{LIndexFormula}. 

Let $J=\sum_{w \in W}(-1)^{l(w)}e^{w\rho}$ be the Weyl denominator and let $V_\lambda$ be the irreducible representation of $G$ with highest weight $\lambda \in \Lambda_+$. There is an isomorphism $R^{-\infty}(G)\xrightarrow{\sim} R^{-\infty}(T)^{W-\tn{anti}}$ given by
\[ \sum_{\lambda \in \Lambda_+}n_\lambda \Tr_{V_\lambda}\mapsto \sum_{\lambda \in \Lambda_+}n_\lambda J\cdot \Tr_{V_\lambda}|_T=\sum_{\lambda \in \Lambda_+}c_\lambda \sum_{w \in W}(-1)^{l(w)}e^{w(\lambda+\rho)}. \]
Let $I^G_T$ be the inverse map. Then $I^G_T$ sends a Weyl numerator to the corresponding irreducible character. Equivalently $I^G_T$ is $|W|^{-1}$ times the Dirac induction map. We extend the definition and notation in the obvious way to the products $\sul{G}=G^b$ and $\sul{T}=T^b$. Formally, we define a $\sul{G}$-equivariant index map
\begin{equation} 
\label{e:GequivIndex}
\index^k_{\sul{G}}=I^{\sul{G}}_{\sul{T}}\circ \index^k_{\sul{T}} \colon K^\ad_{\sul{G}}(\M_\Sigma)\rightarrow R^{-\infty}(\sul{G}).
\end{equation}
\begin{corollary}
The $R(\sul{G})$-subalgebra of $K_{\sul{G}}(\M_\Sigma)=K^0_{\sul{G}}(\M_\Sigma)$ generated by the Atiyah-Bott classes $\sf{E}^NV$, $\dim(N)=0,1,2$ is contained in $K^\ad_{\sul{G}}(\M_\Sigma)$. Therefore for any $\sf{E}$ in this subalgebra and for any $k>-\hvee$ the index
\[ \index^k_{\sul{G}}(\sf{E})\in R^{-\infty}(\sul{G}) \]
is well-defined.
\end{corollary}

\section{The Chern character and the Kirillov-Berline-Vergne formula}
In this section we introduce a different description of the K-theory class $\sf{E}^\Sigma V$ based on gluing disks to each component of the boundary $\partial \Sigma$ to obtain a closed surface $\bSigma$. We then construct a closed equivariant differential form representing the equivariant Chern character of $\sf{E}^\Sigma V$. This is used to prove a Kirillov-Berline-Vergne formula for the index.

\subsection{$\sf{E}^\Sigma V$ revisited}
For the purpose of this section, it is convenient to work with the presentation $\M_\Sigma=\A_{\Sigma,\C}^\flat/\G_{\Sigma,\C}$ mentioned briefly in Section \ref{s:modsp} and reviewed here; this presentation leads, along the same lines as Section \ref{s:familyAPS}, to a canonically isomorphic family of Dirac operators parametrized by $\M_\Sigma$, hence nothing is lost in passing between presentations.

Let $\C\simeq \partial \Sigma \times [1,2]$ denote a collar neighborhood of the boundary, where the boundary sits at $r=1$ in terms of the coordinate $r \in [1,2]$. Recall that the boundary $\partial \Sigma$ has $b\ge 1$ parametrized components, denoted $\partial_1 \Sigma,...,\partial_b \Sigma$. Let $\C_j\simeq \partial_j \Sigma \times [1,2]$ denote the $j$-th component of the collar. Let $r_j=r|_{\C_j}$. Let $s_j \colon \partial_j \Sigma \xrightarrow{\sim} S^1=\bR/\bZ$ denote the parameterization of $\partial_j \Sigma$, and recall $(\partial_{r_j},\partial_{s_j})$ is an oriented frame along $\partial_j\Sigma$ since $\partial_{r_j}$ is interior-pointing. Denote the pullback of $A \in \A_\Sigma^\flat$ to $\partial_j \Sigma$ by
\begin{equation}
\label{e:defAj}
A_j(s_j)\d s_j,
\end{equation}
where $A_j \in \Omega^0_{\varsigma-1}(S^1,\g)$. 

The subspace $\A_{\Sigma,\C}^\flat \subset \A_\Sigma^\flat$ denotes flat connections $A$ on $\Sigma$ such that for $j=1,...,b$,
\begin{equation} 
\label{e:AflC}
A|_{\C_j}=A_j(s_j)\d s_j.
\end{equation}
In other words $A|_{\C_j}$ has no component normal to $\partial_j \Sigma$, and the coefficient of $\d s_j$ only depends on $s_j$. The subgroup $\G_{\Sigma,\C}\subset \G_{\Sigma,\partial \Sigma}$ denotes gauge transformations $g$ such that $g|_{\C}=1$. As mentioned in Section \ref{s:modsp}, the loop group space $\M_\Sigma$ is canonically isomorphic to the quotient $\A_{\Sigma,\C}^\flat/\G_{\Sigma,\C}$.

Let $\bD$ denote the unit disk in $\bR^2$. Let $\bD_1,...,\bD_b$ denote $b$ copies of $\bD$, with polar coordinates $(r_j,s_j) \in [0,1]\times \bR/\bZ$ on the $j$-th copy. Define the closed surface $\bSigma$ by capping off the boundary:
\begin{equation} 
\label{e:bsigma}
\bSigma=\Sigma \cup_{\partial \Sigma} \bigcup_{j=1}^b \bD_j.
\end{equation}
The coordinates $r_j$, $s_j$ are thus defined on the subset $\C_j \cup_{\partial_j \Sigma} \bD_j$ of $\bSigma$.

Elements of $\A_{\Sigma,\C}^\flat$ can be extended to non-flat connections on $\bSigma$, such that smooth connections are extended to smooth connections, as follows. Fix a bump function $\chi(r)$ equal to $1$ for $r>\frac{2}{3}$ and equal to $0$ for $r<\frac{1}{3}$. For $A \in \A_{\Sigma,\C}^\flat$, let $\mathbb{A}$ be the connection obtained by extending $A$ constantly over each punctured disk $\bD_j\backslash \{0\}$ and multiplying by $\chi(r_j)$ in order that the product extends by $0$ to $\bD_j$:
\begin{equation} 
\label{e:extension}
\mathbb{A}|_{\C_j\cup \bD_j}(r_j,s_j)=A_j(s_j)\chi(r_j)\d s_j, \qquad r_j \in [0,2], \quad s_j \in \partial_j \Sigma\simeq S^1.
\end{equation}
Let $\G_{\Sigma,\C-lc} \subset \G$ be the subgroup consisting of gauge transformations that are locally constant on $\C$. Elements of $\G_{\Sigma,\C-lc}$ have a canonical extension to $\bSigma$ constant on each of the caps. In particular elements of $\G_{\Sigma,\C} \subset \G_{\Sigma,\C-lc}$ are extended to $\bSigma$ by the identity.

Extend the Riemannian metric and spinor module to $\bSigma$. Let $\DiracbS_\bA$ denote the Dirac operator acting in $L^2(\bSigma,S^{\bSigma}\boxtimes V)$ constructed using the connection $\bA$. The family $\DiracbSP_\bullet=(\DiracbSP_\bA)_{A \in \A_{\Sigma,\C}^\flat}$ is $\G_{\Sigma,\C-lc}$-equivariant, hence descends to a $\G_{\Sigma,\C-lc}/\G_{\Sigma,\C}\simeq \sul{G}$-equivariant family of Fredholm operators $\DiracbSP_{q(\bullet)}$ over $\M_\Sigma$.

\begin{proposition}
\label{p:cappedclass}
The class in $K_{\sul{G}}(\M_\Sigma)$ defined by the family $\DiracbSP_{q(\bullet)}$ is $\sf{E}^\Sigma V$.
\end{proposition}
\begin{proof}
This is a consequence of a families version of the splitting theorem (cf. \cite[Theorem 8.17]{BarBallmann}, and see for example \cite{dai1996splitting} for a families version which is more sophisticated than is really needed here): the K-theory class defined by the family of Dirac operators $\DiracbSP_\bA$ is the sum of the class $\sf{E}^\Sigma V$ defined by the family of boundary problems $(\DiracSP_A,B_{<0}(\partial))$, and the classes defined by the family of boundary problems $(\Dirac{}^{\bD_j,+}_{\bA|_{\bD_j}},B_{\le 0}(\partial))$, $j=1,...,b$. Each of the latter boundary problems on the disk is the pullback under the moment map of a tautological family of boundary problems parametrized by $L\g^*$. But $K^0_G(L\g^*)=R(G)$ because $L\g^*$ is $G$-equivariantly contractible. To determine which element of $R(G)$, it suffices to compute the index of a single operator in the family, which is $0$ (see for example the discussion at the end of Section \ref{s:familyAPS}).
\end{proof}

Thus using the family $\DiracbSP_{q(\bullet)}$ leads to the same class $\sf{E}^\Sigma V \in K_{\sul{G}}(\M_\Sigma)$. The loop group space $\M_\Sigma$ has a finite dimensional (singular) symplectic quotient $\mu_{\M_\Sigma}^{-1}(0)/\sul{G}$, and Proposition \ref{p:cappedclass} also shows that the class induced by $\sf{E}^\Sigma V$ on this finite dimensional moduli space coincides with the corresponding classes introduced by Atiyah and Bott \cite{AtiyahBottYangMills}.

\subsection{Chern character forms}\label{s:ChChar}
Recall that we defined (Definition \ref{d:EV}) an $L\sul{G}$-equivariant vector bundle $\sf{E}V\rightarrow \M_\Sigma \times \Sigma$. 
\begin{definition}
Let $\bE V \rightarrow \M_\Sigma \times \bSigma$ be the vector bundle
\[ \bE V=(\A_{\Sigma,\C}^\flat\times \bSigma \times V)/\G_{\Sigma,\C} \rightarrow \M_\Sigma\times \bSigma. \] 
\end{definition}
Unlike $\sf{E}V$, the vector bundle $\bE V$ is only $\G_{\Sigma,\C-lc}/\G_{\Sigma,\C}=\sul{G}$-equivariant. The restriction of $\bE V$ to $\M_\Sigma\times \Sigma \subset \M_\Sigma\times \bSigma$ is canonically identified with $\sf{E}V$. We will use the same symbol $\bE V$ (resp. $\sf{E}V$) to denote the $N(\sul{T})$-equivariant (resp. $N(\sul{T})\ltimes \ul{\Pi}$-equivariant) vector bundle obtained by restricting $\bE V$ to $X \times \bSigma$ (resp. restricting $\sf{E}V$ to $X \times \Sigma$). 

The restriction of $\bE V$ to $X \times (\C \cup \bD_1 \cup \cdots \cup \bD_b)$ is canonically trivial since $\G_{\Sigma,\C}$ acts trivially over this subset of $\bSigma$. Let $\nabla^{\sf{E}V}$ be a $N(\sul{T})\ltimes \ul{\Pi}$-invariant connection on the vector bundle $\sf{E}V=\bE V|_{X\times \Sigma}$. Over the collar and relative to the trivialization, $\nabla^{\sf{E}V}|_{X\times \C_j}=\d+\alpha_j$ for some connection $1$-form $\alpha_j$ on $X \times \C_j$. The trivialization of $\sf{E}V|_{X\times\C_j}$ is not compatible with the $\ul{\Pi}$-action: under the action of $\ul{\eta}=(\eta_1,...,\eta_b)\in \ul{\Pi}$, the de Rham differential $\d=\d_X+\d_{r_j}+\d_{s_j}$ on the collar transforms as 
\begin{equation} 
\label{e:deRhamPi}
\exp(s_j\eta_j)(\d_X+\d_{r_j}+\d_{s_j})\exp(-s_j\eta_j)=\d_X+\d_{r_j}+\d_{s_j}-\eta_j. 
\end{equation}
Therefore to obtain a $\ul{\Pi}$-invariant connection, the $1$-form $\alpha_j$ cannot be $0$. Perhaps the simplest choice is to take $\alpha_j$ of the form
\begin{equation}
\label{e:alphaj}
\alpha_j(x,r_j,s_j)=\mu_j(x)f(s_j)\d s_j,
\end{equation} 
where $f(s_j)$ is any smooth function on the circle with total integral $1$ (we could choose $f=1$ but the additional flexibility will be convenient later on). Since $\mu_j(\ul{\eta}\cdot x)=\mu_j(x)+\eta_j$, equation \eqref{e:deRhamPi} shows that the connection $\d+\alpha_j$ on $\bE V|_{X\times \C_j}$ is $\ul{\Pi}$-invariant.

Let $\nabla^{\bE V}$ be the $N(\sul{T})$-invariant connection on $\bE V$ extending $\nabla^{\sf{E}V}$ given by $\d+\chi(r_j)\alpha_j$ over $X\times (\C_j \cup \bD_j)$, where from now on we assume $\alpha_j$ is as in \eqref{e:alphaj}.

The Chern-Weil construction yields equivariant Chern character forms $\Ch^{\ul{u}}(\bE V,\ul{\xi}) \in \Omega(X^{\ul{u}}\times \bSigma)$, for $\ul{u} \in \sul{T}$, $\ul{\xi} \in \ul{\t}$. Define $\Ch^{\ul{u}}(\sf{E}^\Sigma V,\ul{\xi})\in \Omega(X^{\ul{u}})$ by integrating over the fibers:
\begin{equation}
\label{e:ChForm1}
\Ch^{\ul{u}}(\sf{E}^\Sigma V,\ul{\xi})=\int_{\bSigma} \Ch^{\ul{u}}(\bE V,\ul{\xi}). 
\end{equation}
Since $\Ahat(\bSigma)=1$, the equivariant Atiyah-Singer families index theorem over a compact base \cite{AtiyahSingerIV}, \cite[Chapters 8, 10]{BerlineGetzlerVergne} shows that the pullback of $\Ch^{\ul{u}}(\sf{E}^\Sigma V,\ul{\xi})$ to any relatively compact open $\sul{T}$-invariant $U \subset X^{\ul{u}}$ represents the equivariant Chern class of $\sf{E}^\Sigma V|_U$.
\begin{theorem}
\label{t:ChernFormProperties}
There is a compact subset $K\subset \t$ and a constant $C>0$ such that for all $\ul{u}\in \sul{T}$ and all $x \in X$, the closed equivariant differential form $\Ch^{\ul{u}}(\sf{E}^\Sigma V,\ul{\xi})\in \Omega(X^{\ul{u}})$ defined in \eqref{e:ChForm1} satisfies:
\begin{enumerate}[i)]
\item the Fourier transform of the vector-valued function
\[ \ul{\xi} \in \ul{\t}\mapsto \Ch^{\ul{u}}(\sf{E}^\Sigma V,\ul{\xi})_x \in \wedge T^*_x X^{\ul{u}} \]
has support contained in $K^b\subset \ul{\t}$;
\item there is an estimate
\[ |\Ch^{\ul{u}}(\sf{E}^\Sigma V,\ul{\xi})_x|\le C(1+|\mu(x)|). \]
\end{enumerate}
The contribution of the caps $\bD_1,...,\bD_b$ to $\Ch^{\ul{u}}(\sf{E}^\Sigma V,\ul{\xi})$ is the $0$-form
\begin{equation}
\label{e:capcontrib}
-(2\pi \i)^{-1}\sum_{j=1}^b \partial_{\mu_j} \Tr_V(u_j\exp(\xi_j))
\end{equation}
where $\ul{u}=(u_1,...,u_b)$, $\ul{\xi}=(\xi_1,...,\xi_b)$ and $\mu=(\mu_1,...,\mu_b)$. Under the action of $\ul{\eta} \in \ul{\Pi}$,
\begin{equation} 
\label{e:etagrad}
\ul{\eta}^*\Ch^{\ul{u}}(\sf{E}^\Sigma V,\ul{\xi})=\Ch^{\ul{u}}(\sf{E}^\Sigma V,\ul{\xi})-(2\pi \i)^{-1}\partial_{\ul{\eta}} \ul{\Tr}_V(\ul{u}\exp(\ul{\xi})).
\end{equation}
\end{theorem}
\begin{proof}
The fiber integral over $\bSigma$ can be split into integrals over $\Sigma$, and $\bD_1,...,\bD_b$. Since $\nabla^{\bE V}\upharpoonright X \times \Sigma$ is $\ul{\Pi}$-invariant, the contribution of the integral over $\Sigma$ is $\ul{\Pi}$-invariant hence easily satisfies properties i) and ii). On the other hand, over $X\times \bD_j$ we have an explicit formula for the connection 1-form ($\chi(r_j)\mu_j(x)f(s_j)\d s_j$), with curvature
\[ F_j(x,r_j,s_j)=\chi'(r_j)\mu_j(x)f(s_j)\d r_j\d s_j+\chi(r_j)f(s_j)\d \mu_j(x)\d s_j . \]
At the point $(x,r_j,s_j)\in X^{\ul{u}} \times \bD_j$, the form $\Ch^{\ul{u}}(\bE V,\ul{\xi})$ is $\Tr_V(u_j\exp(\xi_j)\exp(\tfrac{\i}{2\pi}F_j(x,r_j,s_j)))$. Integrating over $\bD$, the term involving $\d\mu_j(x)$ does not contribute since it does not contain a factor of $\d r_j$, while the integral of $\chi'(r_j)\mu_j(x)f(s_j)\d r_j\d s_j$ is $\mu_j(x)$. Hence the contribution of $\bD_j$ to the integral over the fiber at $x \in X^{\ul{u}}$ is 
\begin{equation} 
\label{e:afterint}
\Tr_V(u_j\exp(\xi_j)\tfrac{\i}{2\pi} \mu_j(x))=-(2\pi \i)^{-1}\partial_{\mu_j(x)}\Tr_V(u_j\exp(\xi_j)).
\end{equation}
To deduce i), note that the $\t$-Fourier transform of \eqref{e:afterint} (i.e. the Fourier transform in the variable $\xi_j$) is in any compact ball $K \subset \t^*$ containing the weights of the finite dimensional $T$-representation $V$. The estimate in ii) follows from \eqref{e:afterint} since $|\mu_j(x)|\le |\mu(x)|$. Equation \eqref{e:etagrad} is an immediate consequence of \eqref{e:afterint} and $\mu_j(\ul{\eta}x)=\mu_j(x)+\eta_j$, for $\ul{\eta}=(\eta_1,...,\eta_b)\in \ul{\Pi}$.
\end{proof}

\subsection{Kirillov-Berline-Vergne formula}
Theorem \ref{t:ChernFormProperties} (as well as the remarks in the paragraph preceding the proposition) show that the hypotheses of \cite[Theorem 4.5]{LIndexFormula} are satisfied, and therefore we have a Kirillov-Berline-Vergne formula (equation (43) in \cite{LIndexFormula}) for $\index^k_{\sul{T}}(\sf{E}^\Sigma V)$ in terms of integration of equivariant differential forms over fixed-point submanifolds $X^{\ul{u}}\subset X$. For the special case $\ul{u}=1$ the formula states that for $\ul{\xi}\in \ul{\t}$ sufficiently small,
\begin{equation} 
\label{e:KirillovBerlineVergne}
\index^k_{\sul{T}}(\sf{E}^\Sigma V)(\exp(\ul{\xi}))=\int_X \Ahat(X,\ul{\xi})\Ch(\sf{E}^\Sigma V,\ul{\xi})\Ch(\scr{L}_k,\ul{\xi})^{1/2}\Ch(\scr{B},\ul{\xi}),
\end{equation}
where $\scr{L}_k=\scr{L}_{\tn{can}}\otimes L^{2k}$ is the line bundle \eqref{e:anticanon} and $\Ch(\scr{B},\ul{\xi})$ is the pullback of an equivariant Chern form for the Thom-Bott element \eqref{e:ThomBott}. There is a similar formula for general $\ul{u}$. The integral converges in the sense of distributions, i.e. if $\chi(\ul{\xi}) \in C^\infty_c(\ul{\t})$ has support sufficiently close to $0 \in \ul{\t}$, then the integral over $X$ converges after pairing the integrand with the test function $\chi(\ul{\xi})$.

Equation \eqref{e:KirillovBerlineVergne} is deduced in \cite{LIndexFormula} starting from the (K-theory) non-abelian localization formula for the index. A theorem of Paradan-Berline-Vergne can be applied to obtain a cohomological formula for the contribution of each component of the critical set of the norm-square of the moment map. Having passed to cohomology, the contributions can be re-summed, now using non-abelian localization in cohomology (in the reverse direction that it is usually applied), resulting in \eqref{e:KirillovBerlineVergne}.

\section{The fixed point formula}
In \cite[Theorem 4.11]{LIndexFormula} we proved an Atiyah-Bott-Segal-Singer-type fixed point formula for the index of suitable generating series of admissible K-theory classes on a Hamiltonian loop group space. Here we shall calculate the fixed-point contributions for the Atiyah-Bott classes.

\subsection{Application of the fixed point formula}
Let $t$ be a formal variable. Following Teleman and Woodward \cite{TelemanWoodward}, define a formal series of K-theory classes
\begin{equation}
\label{e:defformalseries}
\sf{E}_t=\exp(t\sf{E}^\Sigma V)=\sum_{n=0}^\infty \frac{t^n}{n!}(\sf{E}^\Sigma V)^{\otimes n} \in \bC K^{\ad}_{\sul{G}}(\M_\Sigma)\ft,
\end{equation}
where $\bC K^{\ad}_{\sul{G}}(\M_\Sigma)\ft=\bC\otimes K^{\ad}_{\sul{G}}(\M_\Sigma)\ft$, and $K^{\ad}_{\sul{G}}(\M_\Sigma)\ft$ denotes the $R(\sul{G})\ft$-algebra of formal series (in $t$) whose coefficients are admissible K-theory classes. The equivariant index of such a series is the element of $\bC R^{-\infty}(T)\ft$ obtained by taking the index term by term. By Theorem \ref{t:ChernFormProperties} the corresponding formal series of Chern character forms satisfies
\begin{equation} 
\label{e:twistedequivCh}
\ul{\eta}^*\Ch^{\ul{u}}(\sf{E}_t,\ul{\xi})=e^{-t(2\pi \i)^{-1}\partial_{\ul{\eta}} \ul{\Tr}_V(\ul{u}\exp(\ul{\xi}))}\Ch^{\ul{u}}(\sf{E}_t,\ul{\xi}), \quad \text{for}\quad \ul{\eta}\in \ul{\Pi}.
\end{equation}
In the terminology of \cite[Definition 4.7 and Section 4.6]{LIndexFormula}, equation \eqref{e:twistedequivCh} says that $\sf{E}_t$ admits infinitesimally twisted $N(\sul{T})\ltimes \ul{\Pi}$-equivariant Chern character forms, with infinitesimal twist 
\[ v=\tfrac{1}{4\pi^2}\nabla \ul{\Tr}_V \in (\ul{\t}_\bC \otimes \bC R(\sul{T})\ft )^{\sul{W}}, \]
where the gradient $\nabla$ is with respect to the basic inner product. Equation \eqref{e:twistedequivCh} and Theorem \ref{t:ChernFormProperties} show that the hypotheses of \cite[Theorem 4.11]{LIndexFormula} are satisfied, and a direct application of the latter yields the index formula:
\begin{equation}
\label{e:abelianlocAB0}
\index^k_{\sul{T}}(\sf{E}_t)=\frac{1}{|\sul{T}_\ell|}\sum_{\sul{g} \in \sul{T}_\ell^\reg/\sul{W}}\frac{\sum_{\sul{w}\in \sul{W}}(-1)^{l(\sul{w})}\delta_{\sul{w}\sul{g}_t}}{\det(1-t\ell^{-1}\sul{H}_V(\sul{g}_t))}\int_{X^{\sul{g}}/\ul{\Pi}} \exp\big(t\Ch^{\sul{g}_t}(\sf{E}^\Sigma V)\big)\A\S^{\sul{g}_t}(\sigma).
\end{equation}
Both sides of \eqref{e:abelianlocAB0} are formal series (in $t$) of $\sul{W}$-antisymmetric distributions on $\sul{T}$. The set $\sul{T}_\ell^{\reg}$ is the subset of regular elements of the finite subgroup $\sul{T}_\ell=\ell^{-1}\ul{\Lambda}/\ul{\Pi}$ of $\sul{T}$. For $\sul{g} \in \sul{T}_\ell^\reg$, $\sul{g}_t \in \J_{\sul{g}}^\infty \sul{T}_\bC$ is the unique $\infty$-jet (or $\bC \ft$-point) based at $g$ that solves the fixed point equation $\Phi_t(\sul{g}_t)=g$ and $\Phi_t(\ul{u})=\ul{u}\exp(t\ell^{-1}v(\ul{u}))$, $v=\tfrac{1}{4\pi^2}\nabla \ul{\Tr}_V$. In the numerator $\delta_{\sul{w}\sul{g}_t}$ is the formal series of distributions supported at the point $\sul{w}\sul{g}$ that acts on a character $f \in \bC R(\sul{T})$ by evaluation
\[ f\mapsto f(\sul{w}\sul{g}_t) \in \bC\ft,\] 
defined using Taylor expansion of $f$, viewed as a function on $T_\bC$.

The full expression is independent of the choices of representatives of orbits for $\sul{T}_\ell^\reg \circlearrowleft \sul{W}$, but for later convenience we choose the sum to run over the unique set of representatives in $(\sul{T}_\ell)^\reg \cap \exp(\ul{\mf{A}})$, where $\ul{\mf{A}}\subset \ul{\t}_+$ is the fundamental alcove. In the denominator $\sul{H}_V(\ul{u})=\d v(\ul{u})\in \End(\ul{\t}_\bC)$ is $1/4\pi^2$ times the Hessian of $\ul{\Tr}_V$ at $\ul{u}$. The differential form $\A\S^{\sul{g}_t}(\sigma)$ is the Atiyah-Singer integrand associated to the symbol $\sigma$ obtained by Bott-Thom twist of the symbol of the Dirac operator $\Dirac_{\scr{S}_k}$ (of course the latter can be written in terms of familiar characteristic forms). 

In \cite{LIndexFormula}, the index formula \eqref{e:abelianlocAB0} is deduced from the Kirillov-Berline-Vergne formula \eqref{e:KirillovBerlineVergne} using Theorem \ref{t:ChernFormProperties}. The integral over the non-compact manifold $X$ (or $X^{\ul{u}}$ more generally) is replaced with an integral over a fundamental domain for the action of $\ul{\Pi}$ and a sum over $\ul{\Pi}$. The latter sum is evaluated using the Poisson summation formula and the change of variables $\Phi_t$. Further localization (abelian localization for integrals of closed equivariant forms) and simplification leads to \eqref{e:abelianlocAB0}.

Formulas below will be slightly simpler if we work with the corresponding $\sul{G}$-equivariant index \eqref{e:GequivIndex}, which is an element of $\bC R^{-\infty}(\sul{G})\ft^{\sul{G}}$. Equation \eqref{e:abelianlocAB0} (cf. \cite[Equation (66)]{LIndexFormula}) yields
\begin{equation}
\label{e:abelianlocAB}
\index^k_{\sul{G}}(\sf{E}_t)=\frac{1}{|\sul{T}_\ell|}\sum_{\sul{g} \in \sul{T}_\ell^\reg/\sul{W}}\frac{(-1)^{|\ul{\mf{R}}_+|}\sul{J}(\sul{g}_t)\delta_{\sul{g}_t}}{\det(1-t\ell^{-1}\sul{H}_V(\sul{g}_t))}\int_{X^{\sul{g}}/\ul{\Pi}} \exp\big(t\Ch^{\sul{g}_t}(\sf{E}^\Sigma V)\big)\A\S^{\sul{g}_t}(\sigma).
\end{equation}
Here $\delta_{\sul{g}_t}$ is the formal series of distributions sending $f \in \bC R(\sul{G})$ to $f(\sul{g}_t)$. We aim to evaluate the integral \eqref{e:abelianlocAB}. This will require some preparation.

\subsection{The fixed point locus}
Recall from Section \ref{s:globtrans} that $X/\ul{\Pi} \subset M_\Sigma=\M_\Sigma/\Omega \sul{G}$, where $M_\Sigma$ is the holonomy manifold of $\M_\Sigma$ introduced in Section \ref{s:modsp}. We shall use a parametrization of $M_\Sigma$ similar to \cite[Theorem 9.3]{AlekseevMalkinMeinrenken}. Let $p_j \in \partial_j \Sigma\simeq S^1=\bR/\bZ$ be the basepoint on the $j$-th boundary component, and recall
\[ P=\{p_1,...,p_b\} \subset \partial \Sigma.\]
Let $a_1,b_1,...,a_\sg,b_\sg$ be a standard set of generators for $\pi_1(\bSigma,p_1)$ with images contained in $\Sigma$; $c_2,...,c_b$ a set of paths connecting $p_1$ to each of $p_2,...,p_b$; and $d_2=\partial_2 \Sigma,...,d_b=\partial_b \Sigma$ the boundary circles. We may arrange that all of these curves are smooth, simple and disjoint except for intersections of the $a_i,b_i,c_j$ at $p_1$ and the intersections $c_j\cap d_j=\{p_j\}$. The set $\{a_i,b_i,c_j,d_j\}$ generates the fundamental groupoid $\Pi_1(\Sigma,P)$ of paths in $\Sigma$ with endpoints contained in $P$.

Let $\G_{\Sigma,P}\subset \G_\Sigma$ denote gauge transformations equal to the identity over the subset $P$. The holonomy manifold may be identified with a space of representations of the fundamental groupoid
\[ M_\Sigma=\A_\Sigma^\flat/\G_{\Sigma,P}\simeq \Hom(\Pi_1(\Sigma,P),G).\] 
The system of generators $\{a_i,b_i,c_j,d_j\}$ determines a diffeomorphism $h\colon M_\Sigma\xrightarrow{\sim} G^{2(\sg+b-1)}$ given by
\begin{equation}
\label{e:holcoord}
[A]\mapsto h([A])=(h_{a_1}(A),h_{b_1}(A),...,h_{a_{\sg}}(A),h_{b_\sg}(A),h_{c_2}(A),h_{d_2}(A),...,h_{c_b}(A), h_{d_b}(A))
\end{equation}
where $h_f(A)$ is the holonomy of the connection $A$ along the path $f$. The action of $\sul{g}=(g_1,...,g_b)\in \sul{G}$ is
\[ h_{a_i}\mapsto g_1h_{a_i}g_1^{-1}, \quad h_{b_i}\mapsto g_1h_{b_i}g_1^{-1}, \quad h_{c_j}\mapsto g_1h_{c_j}g_j^{-1}, \quad h_{d_j}\mapsto g_jh_{d_j}g_j^{-1}.\]
It is convenient to identify $G$ with the diagonal in $\sul{G}$. For $\sul{g}\in \sul{T}_\ell^\reg \cap \exp(\ul{\mf{A}})$, one verifies easily that $M_\Sigma^{\sul{g}}$ is empty unless $g_1=g_2=\cdots=g_b=:g$, and in this case $M_\Sigma^{\sul{g}}=M_\Sigma^g=M_\Sigma^T$ does not depend on the choice of $\sul{g}\in \sul{T}_\ell^{\reg}$ by regularity. Since this space will appear frequently below, we make the following definition.
\begin{definition}
Let $F=M_\Sigma^T$ be the fixed-point set of the diagonal action of $T$ on $M_\Sigma$. Equivalently $F=\Hom(\Pi_1(\Sigma,P),T)$ is the moduli space of $T$-connections with framing over $P$. Under the diffeomorphism $h\colon M_\Sigma \xrightarrow{\sim} G^{2({\rm g}+b-1)}$, $F$ is mapped to the torus $T^{2({\rm g}+b-1)}$.
\end{definition} 
With this observation, we can largely eliminate the `underline' notation in equation \eqref{e:abelianlocAB},
\begin{equation}
\label{e:abelianlocAB2}
\index^k_{\sul{G}}(\sf{E}_t)=\frac{1}{|T_\ell|^b}\sum_{g \in T_\ell^\reg/W}\frac{(-1)^{b|\mf{R}_+|}J(g_t)^b\delta_{g_t}}{\det(1-t\ell^{-1}H_V(g_t))^b}\int_F \exp\big(t\Ch^{g_t}(\sf{E}^\Sigma V)\big)\A\S^{g_t}(\sigma),
\end{equation}
where $g_t \in \J^\infty_g(T_\bC)$ satisfies $\Phi_t(g_t)=g$, and $\Phi_t(u)=u\exp(t\ell^{-1}v(u))$, $v(u)=\tfrac{1}{4\pi^2}\nabla \Tr_V(u)$.

In terms of the non-compact manifold $X$, $F\simeq X^T/\ul{\Pi}$ and $X^T$ can be identified with the fiber product
\[ X^T=\ul{\t}\times_{\sul{T}} F=\ul{\t}\times_{\sul{T}}\Hom(\Pi_1(\Sigma,P),T), \]
with the moment map $\mu=(\mu_1,...,\mu_b)$ being the projection to $\ul{\t}$, and the map to $\sul{T}$ being the map that associates to an equivalence class of a connection its list of holonomies around the boundary circles. For an \emph{abelian} flat connection on $\Sigma$, the product of the holonomies around the boundary circles is $1$, and it follows that the sum
\begin{equation}
\label{e:defbarmu}
\bar{\mu}:=\sum_{i=1}^b\mu_i \colon X \rightarrow \t
\end{equation}
takes values in $\Pi$. In fact the connected components of $X^T$ are the fibers $\bar{\mu}^{-1}(\eta)$ as $\eta$ ranges over the lattice $\Pi$.

On the fixed point set $X^T$, the Atiyah-Segal-Singer integrand $\A\S^{g_t}(\sigma)$ simplifies to\ignore{\footnote{The problem here is that we need to use the holomorphic extension! I forgot about that when making the previous edits. So we shouldn't be using $\ol{J(g_z)}$ which isn't holomorphic, but rather $J(g_z^{-1})$. Or else we should go back to using $\Delta$. In the $SU(2)$ case we are taking about $J(u)=u-u^{-1}$ and the product we want in the denominator is $(u-u^{-1})(u^{-1}-u)$ which one could write as $-(u-u^{-1})^2$, but maybe nicer looking is to write it as $4\sin^2(\theta)$ where $u=e^{\i\theta}$.}}
\begin{equation} 
\label{e:ASSsimplification}
\frac{J(g_t)^b}{(-1)^{|\mf{R}_+|(\sg+b-1)}J(g_t)^{2(\sg+b-1)}}g_t^{\ell \bar{\mu}}.
\end{equation}
The main ingredients involved in the calculation \eqref{e:ASSsimplification} (we refer the reader to \cite[Section 5.5]{AMWVerlinde} for details) are (i) $\Ahat(F)=1$; (ii) the fixed-point set $X^T$ is contained in the $0$-fiber of the map $X\xrightarrow{\mu_\Sigma}\sul{R}\rightarrow \ul{\t}^\perp$, so the Chern character $\Ch^{g_t}(\scr{B})$ of the Bott-Thom element (which occurs as a factor in the Atiyah-Singer integrand $\A\S^{g_t}(\sigma)$) may be replaced by its pullback to $0 \in \ul{\t}^\perp$ which is $J(g_t)^b$; (iii) the normal bundle to the fixed-point manifold is the trivial bundle with fiber $(\g/\t)^{2(\sg+b-1)}$ and contributes a factor $(-1)^{|\mf{R}_+|(\sg+b-1)}J(g_t)^{2(\sg+b-1)}$ in the denominator; the factor $g_t^{\ell\bar{\mu}}$ which comes from the action of $g_t$ on the spinor bundle (note that this factor is absent when $t=0$, since $g\in T_\ell$, $\bar{\mu} \in \Pi$ imply $g^{\ell \bar{\mu}}=1$). With these simplifications \eqref{e:abelianlocAB2} becomes
\begin{equation}
\label{e:abelianlocAB3}
\index^k_{\sul{G}}(\sf{E}_t)=\frac{1}{|T_\ell|^b}\sum_{g \in T_\ell^\reg/W}\frac{(-1)^{|\mf{R}_+|(1-\rm{g})}J(g_t)^{2(1-\rm{g})}\delta_{g_t}}{\det(1-t\ell^{-1}H_V(g_t))^b}\int_F g_t^{\ell \bar{\mu}}e^{\frac{1}{2}c_1(\scr{L})+t\Ch^{g_t}(\sf{E}^\Sigma V)}.
\end{equation}
We evaluate the integral over $F$ in the next section.

\subsection{The Teleman-Woodward formula for $\sf{E}^\Sigma V$}\label{s:TelWooEven}
The integrand in \eqref{e:abelianlocAB3} is defined a priori on $X^T$, but descends to the quotient $F=X^T/\ul{\Pi}$. We verified this in a more general context in  \cite[Theorem 4.9]{LIndexFormula}, but we can also check it directly in this special case. Recall that in Section \ref{s:ChChar} we defined the equivariant Chern character form of $\sf{E}^\Sigma V$ by integration over the fibers $X \times \bSigma \rightarrow X$ of an equivariant Chern character form for $\bE V$. The contribution of $\Sigma \subset \bSigma$ is $\ul{\Pi}$-invariant. On the other hand we calculated the contribution from the caps $\bD_1,...,\bD_b$ in \eqref{e:capcontrib}; when restricted to the diagonally embedded $T\subset \sul{T}$ (as in \eqref{e:abelianlocAB3}), this contribution simplifies to 
\[ -(2\pi \i)^{-1}\partial_{\bar{\mu}}\Tr_V(u\exp(\xi)). \]
Exponentiating, we find that the contribution of the caps $\bD_1,...,\bD_b$ to $e^{t\Ch^{g_t}(\sf{E}^\Sigma V)}$ is a multiplicative factor $\exp(-t(2\pi \i)^{-1}\partial_{\bar{\mu}}\Tr_V(g_t))$. The product
\[ g_t^{\ell \bar{\mu}}\exp(-t(2\pi \i)^{-1}\partial_{\bar{\mu}}\Tr_V(g_t))=g_t^{\ell \bar{\mu}}\exp(2\pi \i tv(g_t)\cdot \bar{\mu})=\Phi_t(g_t)^{\ell \bar{\mu}}=g^{\ell \bar{\mu}}=1 \]
since $\bar{\mu}|_{X^T}$ takes values in $\Pi$ and $g \in T_\ell$. It follows that the integrand in \eqref{e:abelianlocAB3} descends.

It remains to calculate the ($\ul{\Pi}$-invariant) contribution from $\Sigma$. This involves the integral of the Chern character form for the vector bundle $\sf{E}V|_{X^T\times \Sigma}=\bE V|_{X^T \times \Sigma}$ over the fibers $X^T\times \Sigma \rightarrow X$. Since $\sf{E}V$ is $\ul{\Pi}$-equivariant, we can work with the vector bundle 
\begin{equation} 
\label{e:bmEV}
\bm{E}V=\sf{E}V|_{X^T\times \Sigma}/\ul{\Pi}
\end{equation}
over the compact manifold with boundary $F\times \Sigma=(X^T/\ul{\Pi})\times \Sigma$. 

In the discussion below, homology/cohomology groups are taken with $\bC$ coefficients. Recall $P=\{p_1,...,p_b\}$ and note that $\Hom(\Pi_1(\Sigma,P),T)\simeq \Hom(\pi_1(\Sigma/P,P),T)$, where $\pi_1(\Sigma/P,P)$ denotes the fundamental group of the space $\Sigma/P$ based at the point $P$ (in the quotient). Let $(\Sigma/P)^\sim$ be the universal covering space of $\Sigma/P$. Given a character $\lambda \in \Lambda=\Hom(T,U(1))$, let 
\[ h^\lambda \colon F \rightarrow \Hom(\pi_1(\Sigma/P,P),U(1)) \] 
be the composition of $\lambda$ with the holonomy map 
\[ F \xrightarrow{\sim} \Hom(\pi_1(\Sigma/P,P),T).\] 
Define a complex line bundle
\[ \bm{E}\bC_\lambda=(F\times (\Sigma/P)^\sim \times \bC)/\pi_1(\Sigma/P,P)\rightarrow F\times \Sigma/P \]
where $f \in \pi_1(\Sigma/P,P)$ acts on the product by
\[ f\cdot (y,p,z)=(y,f\cdot p,h^\lambda(y,f)\cdot z).\]
Let $\{f_s\}_{s \in S}$ be a set of integral generators of $H_1(\Sigma/P)$ and let $f^s$ be the dual basis of $H^1(\Sigma/P)$ with respect to the canonical pairing. Pairing $h^\lambda$ with $f_s$ yields a map 
\[ \pair{h^\lambda}{f_s}\colon F \rightarrow U(1), \] 
the $\lambda$-component of the holonomy of a connection along $f_s$.
Let $f_s^\lambda \in H^1(F)$ be the pullback of the integral generator of $H^1(U(1))$ under $\pair{h^\lambda}{f_s}$. The $1$-st Chern class of the line bundle $\bm{E}\bC_\lambda$ is
\begin{equation} 
\label{e:Llambda}
c_1(\bm{E}\bC_\lambda)=\sum_s f_s^\lambda f^s \in H^1(F)\hotimes H^1(\Sigma/P)\subset H^2(F\times \Sigma/P).
\end{equation}

We may take the set of generators $\{f_s\}_{s \in S}$ of $\pi_1(\Sigma/P,P)$ to be the images of the elements $\{a_i,b_i,c_j,d_j\}$ introduced in the previous section under the quotient map $\Sigma \rightarrow \Sigma/P$; we shall abuse notation and denote these images by the same symbols. The set $S$ thus contains $2(\sg+b-1)$ elements. Let $a^i,b^i,c^j,d^j \in H^1(\Sigma/P)$ be the dual basis. To obtain differential form representatives, recall that $H^1(\Sigma/P)\simeq H^1(\Sigma,P)$, so the relevant de Rham cohomology group is the relative group 
\[ H^1_{\tn{dR}}(\Sigma,P)=\Omega^1_{\tn{cl}}(\Sigma)/\{\d f\mid f \in \Omega^0(\Sigma),\, f|_P \text{ is constant}\}. \]
Represent the classes $a^i$, $b^i$ by Thom forms supported in tubular neighborhoods of smooth 1-cycles contained in $\Sigma \backslash \C$ which are homotopic to $b_i$, $a_i$ respectively. Represent the classes $c^j$, $d^j$ by Thom forms supported in tubular neighborhoods of the curves $d_j$, $c_j$ respectively. We will abuse notation and use the same symbols $a^i,b^i,c^j,d^j$ for the differential form representatives. Using the pullback of the standard $1$-form on $U(1)$, we also obtain differential form representatives for the $f_s^\lambda=a_i^\lambda,b_i^\lambda,c_j^\lambda,d_j^\lambda$. By \eqref{e:Llambda}, the $1$-st Chern form of $\bm{E}\bC_\lambda$ is
\[ c_1(\bm{E}\bC_\lambda)=\sum_{i=1}^{\rm{g}} (a_i^\lambda a^i+b_i^\lambda b^i)+\sum_{j=2}^b (c_j^\lambda c^j+d_j^\lambda d^j).\]
The corresponding equivariant Chern character form is 
\begin{equation} 
\label{e:ChernEClambda}
\Ch^u(\bm{E}\bC_\lambda)=u^\lambda e^{c_1(\bm{E}\bC_\lambda)} \in \Omega(F).
\end{equation} 
If  
\[ V=\bigoplus_\lambda n_\lambda \bC_\lambda \]
is a representation of $T$ then
\[ \bm{E}V=\bigoplus_\lambda n_\lambda \bm{E}\bC_\lambda \]
identifies with the vector bundle $\sf{E}V|_{X^T\times \Sigma}/\ul{\Pi}$ in \eqref{e:bmEV}. Let
\begin{equation}
\label{e:ChernEV}
\Ch^u(\bm{E}V)=\sum_\lambda n_\lambda \Ch^u(\bm{E}\bC_\lambda).
\end{equation}
This is a differential form representative of the $T$-equivariant Chern class of $\bm{E}V$. 

Shortly we will use \eqref{e:ChernEV} to calculate the contribution of $\Sigma \subset \bSigma$ to the Chern character form of $\sf{E}^\Sigma V$ and the integral in \eqref{e:abelianlocAB3}. Justification that \eqref{e:ChernEV} can be used for this purpose is called for because $F \times \Sigma$ has non-empty boundary (having the correct de Rham cohomology class is not quite enough): based on the construction of Section \ref{s:ChChar}, it suffices to show that there is a $\sul{T}\times \ul{\Pi}$-invariant connection on the pullback $\pi^*\bm{E}\bC_\lambda$ of $\bm{E}\bC_\lambda$ under the quotient map $\pi \colon X^T\times \Sigma\rightarrow (X^T/\ul{\Pi})\times \Sigma=F\times \Sigma$ that takes the form (relative to the trivialization)
\begin{equation} 
\label{e:bdform}
\d+\pair{\alpha_j(x,r_j,s_j)}{\lambda}=\d+\pair{\mu_j(x)}{\lambda}f(s_j)\d s_j, \quad \text{ where } \quad \int_{S^1} f(s_j)\d s_j=1
\end{equation}
over the collar $X^T \times \C_j$, and such that the corresponding Chern-Weil form is the pullback $\pi^*c_1(\bm{E}\bC_\lambda)$ of the form constructed above. This follows from the following simple observations:
\begin{enumerate}
\item For $j=2,...,b$, we chose the form $d^j$ to be a Thom form supported in a tubular neighborhood of $\gamma_j$ (a smooth curve connecting $p_1 \in \partial_1 \Sigma$ to $p_j\in \partial_j \Sigma$). We may choose $d^j$ to take a product form on $\C_j$, and then $d^j|_{\C_j}=f(s_j)\d s_j$ for a suitable smooth function $f(s_j)$ whose integral over $S^1$ is $1$.
\item The exponential $\exp(\mu_j(x)) \in T$ is the holonomy of the $\G_{\Sigma,\partial \Sigma}$-equivalence class of connections $x=[A]$ around the boundary circle $\partial_j \Sigma$. Hence for $j=2,...,b$, since $d_j^\lambda$ was the pullback of the standard $1$-form on $U(1)$ by the $\lambda$-component of the holonomy map for $\delta_j$, we have $\d \pair{\mu_j}{\lambda}=\pi^*d_j^\lambda$, and so
\[ \d \pair{\alpha_j}{\lambda}=\pi^*d_j^\lambda d^j|_{X^T\times \C_j}.\]
\item We may choose $d^j$ to take a product form on $\C_1$ as well (near the other end of $\gamma_j$), say $-f(s_1)\d s_1$ (a minus sign appears, relative to the formula near $\C_j$, because the orientation of the boundary is induced from that on $\Sigma$). Then
\[ \sum_{j=2}^b \pi^*d_j^\lambda d^j|_{X^T\times \C_1}=-\sum_{j=2}^b \d\pair{\mu_j}{\lambda} f(s_1)\d s_1 \]
but $\sum_{j=1}^b \mu_j$ is locally constant on $X^T$ as we explained above. Thus
\[ \sum_{j=2}^b \pi^*d_j^\lambda d^j|_{X^T\times \C_1}=\d\pair{\mu_1}{\lambda}f(s_1)\d s_1.\]
\item The three items above show that $\sum_s \pi^*f_s^\lambda f^s|_{X^T\times \C_j}=\d \pair{\alpha_j}{\lambda}$ where $\alpha_j$ is as in \eqref{e:bdform}. It follows from the standard construction of a prequantum connection for a closed integral $2$-form that there is a $\sul{T}\times \ul{\Pi}$-invariant connection on $\pi^* \bm{E}\bC_\lambda \rightarrow X^T\times \Sigma$ with $1$-st Chern form $\sum_s \pi^*f_s^\lambda f^s$. Recall that the construction involves choosing trivializations and primitives over the open subsets of a good cover, thus choosing the specific primitives $\alpha_j$ over $X^T \times \C_j$ ensures that the prequantum connection matches \eqref{e:bdform} on $X^T\times \C_j$.
\end{enumerate}

Proceeding to the calculation, the contribution of $\Sigma$ to $\Ch^u(\sf{E}^\Sigma V)$ is the pullback under $\pi \colon X^T \rightarrow X^T/\ul{\Pi}=F$ of
\begin{equation}
\label{e:SigmaContrib}
\int_\Sigma \Ch^u(\bm{E}V)=\sum_\lambda n_\lambda u^\lambda \int_\Sigma \exp\Big(\sum_s f_s^\lambda f^s\Big).
\end{equation}
Since $\dim(\Sigma)=2$ and the $f_s^\lambda$ are pullbacks of forms on $X^T$, the only term that contributes in the integral over $\Sigma$ comes from the quadratic term in the exponential. Thus
\[ \int_\Sigma \Ch^t(\bm{E}V)=\sum_\lambda n_\lambda t^\lambda \frac{1}{2}\int_\Sigma \Big(\sum_s f_s^\lambda f^s\Big)^2.\]
From our description of the forms, it is clear that
\[ \int_\Sigma a^i b^i=1=-\int_\Sigma b^i a^i, \qquad \int_\Sigma c^j d^j=1=-\int_\Sigma d^jc^j,\]
while all other intersection pairings are trivial. Thus \eqref{e:SigmaContrib} yields
\begin{equation} 
\label{e:ChbmEV}
\int_\Sigma \Ch^u(\bm{E}V)=-\sum_\lambda n_\lambda u^\lambda \Big(\sum_i a_i^\lambda b_i^\lambda+\sum_j c_j^\lambda d_j^\lambda\Big).
\end{equation}
(The minus sign appears after integration over the fibers since for example $a_i^\lambda a^i b_i^\lambda b^i=-a_i^\lambda b_i^\lambda a^i b^i$.) Let $\nu_1,...,\nu_{\dim(\t)}$ be a lattice basis of $\Lambda$. Equation \eqref{e:ChbmEV} is bilinear in $\lambda$, and therefore if we decompose $\lambda=\sum \lambda_p \nu_p$ in terms of the basis, then
\[ \int_\Sigma \Ch^u(\bm{E}V)=-\sum_\lambda \sum_{p,q} n_\lambda u^\lambda \lambda_p\lambda_q\Big(\sum_i a_i^p b_i^q+\sum_j c_j^p d_j^q\Big),\]
where $a_i^p=a_i^{\nu_p}$ and likewise for the other forms. The sum
\[ \sum_\lambda n_\lambda u^\lambda \lambda_p\lambda_q=H_V(u)_{pq} \]
is the $(p,q)$-entry of ($1/4\pi^2$ times) the Hessian matrix of $\Tr_V$ at the point $u$ relative to the basis. Therefore \eqref{e:ChbmEV} may be written
\[ \int_\Sigma \Ch^u(\bm{E}V)=\sum_{p,q} -H_V(u)_{pq}\Big(\sum_i a_i^p b_i^q+\sum_j c_j^p d_j^q\Big).\]

The spinor bundle is at level $\ell$, and one shows (cf. \cite[Corollary 3.12]{meinrenken1999cobordism}, \cite[Proposition 5.2]{AMWVerlinde})
\[ \frac{1}{2}c_1(\scr{L})=\ell \sum_p\Big(\sum_i a_i^p b_i^p+\sum_j c_j^p d_j^p\Big) \]
in $H^2(F)$. Thus the integral from \eqref{e:abelianlocAB3} that we need to compute is
\begin{equation} 
\label{e:integralF}
\int_F \exp\Big(\ell \sum_{p,q}(\delta_{pq}-t\ell^{-1}H_V(u)_{pq})\big(\sum_i a_i^p b_i^q+\sum_j c_j^p d_j^q\big)\Big).
\end{equation}
The product
\[ \prod_{p,i}a_i^p b_i^p\prod_{p,j} c_j^p d_j^p \]
is a top degree form on the $2({\rm g}+b-1)\dim(T)$-dimensional torus $F\simeq T^{2({\rm g}+b-1)}$, whose integral is $|T_1|$ (cf. \cite[Proposition 5.2]{AMWVerlinde}, \cite{BismutLabourie}). Therefore a short calculation with \eqref{e:integralF} leads to
\[ \det(1-\ell^{-1}H_V(u))^{{\rm g}+b-1}|T_\ell|^{{\rm g}+b-1}.\]
Substituting the result of these calculations into \eqref{e:abelianlocAB3} gives the following.

\begin{theorem}
\label{t:indexAB}
The index of the generating series $\exp(t\sf{E}^\Sigma V)$ is
\[ \index^k_{\sul{G}}(\exp(t\sf{E}^\Sigma V))=\sum_{g \in T_\ell^\reg/W} \Big(\frac{(-1)^{|\mf{R}_+|}J(g_t)^2}{|T_\ell|\det(1-t\ell^{-1}H_V(g_t))}\Big)^{1-\sg}\delta_{g_t}\]
where $\ell=k+\hvee$, $g_t \exp(t\ell^{-1}v(g_t))=g$, $v=\tfrac{1}{4\pi^2}\nabla \Tr_V$ and $H_V=\d v$.
\end{theorem}

\subsection{The Teleman-Woodward formula for general classes.}
For completeness, in this section we briefly explain how to extend Theorem \ref{t:indexAB} to include the other more elementary Atiyah-Bott classes introduced in Section \ref{s:ABclasses}. For simplicity we restrict to the case of a product of two odd classes.

Let $C_1,C_2$ be smooth simple closed curves in $\Sigma$, and let $U_1,U_2$ be representations of $G$. We consider the product $\sf{E}^{C_1}U_1\cdot \sf{E}^{C_2}U_2 \in K^\ad_{\sul{G}}(\M_\Sigma)$, which is represented by a family of Dirac operators on the $2$-torus $C_1\times C_2 \subset \Sigma \times \Sigma$. The index of $\exp(t\sf{E}^\Sigma V)\cdot\sf{E}^{C_1}U_1\cdot \sf{E}^{C_2}U_2$ is computed using the fixed-point formula as in \eqref{e:abelianlocAB3}. The only change is an additional factor of $\Ch^{g_t}(\sf{E}^{C_1}U_1\cdot \sf{E}^{C_2}U_2)$, and the latter Chern character can be expressed using the families index theorem. The calculations are similar to those carried out for $\sf{E}^\Sigma V$, and in fact much simpler, since the class $\sf{E}^{C_1}U_1\cdot \sf{E}^{C_2}U_2$ is simply the pullback of a K-theory class on the quotient $X/\ul{\Pi}$.

Decompose $U_1,U_2$ into weight spaces:
\[ U_i=\bigoplus_{\lambda} n_{i,\lambda} \bC_\lambda.\]
For $\lambda_1,\lambda_2 \in \Lambda$, and let $\bm{E}\bC_{\lambda_i}\rightarrow F\times \Sigma$ be the corresponding line bundles (using the same notation as in Section \ref{s:TelWooEven}). Let $\bm{E}\bC_{\lambda_1,\lambda_2} \rightarrow F \times \Sigma_{(1)}\times \Sigma_{(2)}$ (subscripts $(1)$, $(2)$ denoting two copies of $\Sigma$) be the pullback of the exterior product $\bm{E}\bC_{\lambda_1}\boxtimes \bm{E}\bC_{\lambda_2}$ to $F \times \Sigma_{(1)}\times \Sigma_{(2)}$, where $F \hookrightarrow F \times F$ is identified with the diagonal. The $1$-st Chern class of this line bundle is
\[ c_1(\bm{E}\bC_{\lambda_1,\lambda_2})=\sum_s f^{\lambda_1}_sf^s_{(1)}+f^{\lambda_2}_sf^s_{(2)}, \]
and the Chern character
\[ \Ch^u(\bm{E}\bC_{\lambda_1,\lambda_2})=u^{\lambda_1+\lambda_2}\exp\Big(\sum_s f^{\lambda_1}_sf^s_{(1)}+f^{\lambda_2}_sf^s_{(2)}\Big).\]
We then compute the slant product with $C_1 \times C_2 \subset \Sigma_{(1)}\times \Sigma_{(2)}$. The only term in the exponential that contributes is the quadratic term. At the same time we sum over weight space decompositions of $U_1,U_2$, which yields:
\begin{equation} 
\label{e:ChE1E2}
\Ch^u(\sf{E}^{C_1}U_1\cdot \sf{E}^{C_2}U_2)=-\sum_{\lambda_1,\lambda_2}n_{1,\lambda_1}n_{2,\lambda_2}u^{\lambda_1+\lambda_2}\sum_{s_1,s_2}f_{s_1}^{\lambda_1}f_{s_2}^{\lambda_2}\pair{f^{s_1}}{C_1}\pair{f^{s_2}}{C_2},
\end{equation}
where the angled brackets denote pairings between cohomology and homology classes. To slightly simplify the formulas we will assume from now on that $C_1$ is dual to $a^r$ and $C_2$ is dual to $b^r$ for some $1\le r \le \rm{g}$. There is no essential loss of generality, because \eqref{e:ChE1E2} only depends on the homology classes of $C_1,C_2$ in $\Sigma$, each of which is a sum of classes dual to the $a^i,b^i,d^j$, and moreover we will see shortly that the contribution of any other combination ($a^r$, $b^{r'}$ with $r\ne r'$; $a^r,a^{r'}$; $b^r,b^{r'}$; $a^r,d^{r'}$; $b^r,d^{r'}$) vanishes (in particular this will provide a consistency check of Proposition \ref{p:oddclasshomotopy}). With this assumption
\[ \Ch^u(\sf{E}^{C_1}U_1\cdot \sf{E}^{C_2}U_2)=-\sum_{\lambda_1,\lambda_2}n_{1,\lambda_1}n_{2,\lambda_2}u^{\lambda_1+\lambda_2} a_r^{\lambda_1}b_r^{\lambda_2}.\]
Decompose $\lambda_1,\lambda_2$ in terms of a lattice basis and note that
\[ \sum_{\lambda_1,\lambda_2}n_{1,\lambda_1}n_{2,\lambda_2}u^{\lambda_1+\lambda_2}\lambda_{1,p'}\lambda_{2,q'} \]
is the $(p',q')$ component of $-(4\pi^2)^{-1}\nabla \Tr_{U_1}\otimes \nabla \Tr_{U_2}$ at the point $u \in T$. Thus
\[ \Ch^u(\sf{E}^{C_1}U_1\cdot \sf{E}^{C_2}U_2)=\frac{1}{4\pi^2}\sum_{p',q'} (\nabla \Tr_{U_1}(u)\otimes \nabla \Tr_{U_2}(u))_{p'q'}a_r^{p'}b_r^{q'}.\]
Including the Chern character of $\sf{E}^{C_1}U_1\cdot \sf{E}^{C_2}U_2$ in \eqref{e:integralF} therefore results in
\begin{equation}
\label{e:integralFgeneral}
\int_F \exp\Big(\sum_{p,q}A_{pq}(u)\big(\sum_i a_i^p b_i^q+\sum_j c_j^p d_j^q\big)\Big)\frac{1}{4\pi^2}\sum_{p',q'} (\nabla \Tr_{U_1}(u)\otimes \nabla \Tr_{U_2}(u))_{p'q'}a_r^{p'}b_r^{q'},
\end{equation}
where $A_{pq}(u)=\ell \delta_{pq}-tH_V(u)_{pq}$. Inspection of \eqref{e:integralFgeneral} now explains why other combinations of generators of $H^1(\Sigma)$ give a vanishing contribution, for then the integrand has vanishing top degree part. The integral \eqref{e:integralFgeneral} can be calculated similar to a Gaussian integral, and results in
\[ \det(1-\ell^{-1}H_V(u))^{{\rm g}+b-1}|T_\ell|^{{\rm g}+b-1}\alpha^{\ell-tH_V(u)}_{U_1,U_2}(t) \]
where for a matrix $A_{pq}$ with inverse $A^{pq}$ and representations $U_1,U_2$, we have set
\begin{equation} 
\label{e:defalphaA}
\alpha^A_{U_1,U_2}(u)=\frac{1}{4\pi^2}\sum_{p,q}A^{pq}(u)(\nabla \Tr_{U_1}(u)\otimes \nabla \Tr_{U_2}(u))_{pq}. 
\end{equation}
The following theorem summarizes these calculations.

\begin{theorem}
\label{t:indexAB2}
Let $C_1,C_2$ be smooth closed curves in $\Sigma$. Let $U_0,U_1,U_2,V$ be representations of $G$. Then $\index_{\sul{G}}^k(\sf{E}^{\pt}U_0\cdot\sf{E}^{C_1}U_1\cdot\sf{E}^{C_2}U_2\cdot\exp(t\sf{E}^\Sigma V))$ is given by
\[ \sum_{g \in T_\ell^\reg/W} \#(C_1\cap C_2)\Tr_{U_0}(g_t)\alpha^{\ell-tH_V(g_t)}_{U_1,U_2}(g_t)\Big(\frac{(-1)^{|\mf{R}_+|}J(g_t)^2}{|T_\ell|\det(1-t\ell^{-1}H_V(g_t))}\Big)^{1-\sg}\delta_{g_t} \]
where $\ell=k+\hvee$, $g_t \exp(t\ell^{-1}v(g_t))=g$, $v=\tfrac{1}{4\pi^2}\nabla \Tr_V$, $H_V=\d v$, $\#(C_1\cap C_2)$ is the intersection pairing, and $\alpha$ is defined in \eqref{e:defalphaA}.
\end{theorem}

\subsection{The total $\lambda$-operation}\label{s:equivVer}
In previous sections we worked with the exponential generating series $\exp(t\sf{E}^\Sigma V)\in \bC K_T(X)\ft$, and found that the appropriate twist $v=\tfrac{1}{4\pi^2}\nabla \ul{\Tr}_V$ was independent of $t$. Another interesting generating series is the total $\lambda$-operation: 
\begin{equation} 
\label{e:totlambda}
\bm{\lambda}_t \sf{E}^\Sigma V=\sum_{p\ge 0}t^p \bm{\lambda}^p\sf{E}^\Sigma V \in K_T(X)\ft.
\end{equation}
This series appears for example in the application to the equivariant Verlinde formula for the moduli space of Higgs bundles, with $V=\g_\bC$ the adjoint representation \cite{EquivVerlinde, EquivVerlindePf1, EquivVerlindePf2}. 

We first note that if $\sf{E}\in K^\ad_T(X)$ is admissible, then the classes $\bm{\lambda}^p\sf{E}$ (and likewise the Adams operations $\psi^j\sf{E}$) are also admissible. Indeed if $(\E,\nabla^\E,Q)$ is an admissible cycle representing $\sf{E}$, then the various cohomology operations applied to $\sf{E}$ are represented by cycles built by decomposing graded tensor powers $\E^{\hotimes j}$ under the action of the symmetric group (the action having appropriate signs inserted according to the Koszul sign rule; see \cite{atiyah1966power}), and it is clear that the resulting cycles are again admissible.

We briefly explain why \eqref{e:totlambda} leads to a twist $v_t$ that depends on $t$ following \cite[Section 6]{TelemanWoodward}.
\begin{proposition}
\label{p:lambdaseries}
$\sf{E}_t=\bm{\lambda}_t\sf{E}^\Sigma V \in K^\ad_T(X)\ft$ admits infinitesimally twisted Chern character forms with infinitesimal twist
\[ v_t=\frac{1}{4\pi^2}\sum_{j>0}\frac{(-t)^{j-1}}{j^2}\nabla \ul{\Tr}_{\psi^j V}.\]
\end{proposition}
\begin{proof}
One has the formula
\[ \bm{\lambda}_t \sf{E}=\exp\bigg(-\sum_{j>0}\frac{(-t)^j}{j}\psi^j \sf{E}\bigg) \]
in terms of the Adams operations $\psi^j$. By the splitting principle 
\[\Ch^{\ul{u}}(\psi^j \sf{E},\ul{\xi})_{[2p]}=j^p \Ch^{\ul{u}^j}(\sf{E},j\ul{\xi})_{[2p]} \]
where $\alpha_{[2p]}$ is the component of $\alpha$ in degree $2p$. Using the families index formula \eqref{e:ChForm1} for the Chern character form, (and the relation $\psi^j \bE V=\bE (\psi^j V)$),
\[ \Ch^{\ul{u}}(\psi^j\sf{E}^\Sigma V,\ul{\xi})_{[2p]}=j^p\Ch^{\ul{u}^j}(\sf{E}^\Sigma V,j\ul{\xi})_{[2p]}=\frac{1}{j}\int_{\bSigma}j^{p+1}\Ch^{\ul{u}^j}(\bE V,j\ul{\xi})_{[2p+2]}=\frac{1}{j}\Ch^{\ul{u}}(\sf{E}^\Sigma (\psi^j V),\ul{\xi})_{[2p]}.\]
Hence
\[ \Ch^{\ul{u}}(\bm{\lambda}_t \sf{E}^\Sigma V,\ul{\xi})=\exp\bigg(-\sum_{j>0}\frac{(-t)^j}{j^2}\Ch^{\ul{u}}(\sf{E}^\Sigma(\psi^jV),\ul{\xi})\bigg)\]
For $\ul{\eta} \in \ul{\Pi}$, recall that $\ul{\eta}^*\Ch^{\ul{u}}(\sf{E}^\Sigma V,\ul{\xi})=\Ch^{\ul{u}}(\sf{E}^\Sigma V,\ul{\xi})-(2\pi \i)^{-1}\partial_{\ul{\eta}} \ul{\Tr}_V(\ul{u}\exp(\ul{\xi}))$. It follows that
\[ \ul{\eta}^*\Ch^{\ul{u}}(\bm{\lambda}_t \sf{E}^\Sigma V,\ul{\xi})=\exp\bigg(\sum_{j>0}\frac{(-t)^j}{j^2}(2\pi \i)^{-1}\partial_{\ul{\eta}} \ul{\Tr}_{\psi^j V}(\ul{u}\exp(\ul{\xi}))\bigg)\Ch^{\ul{u}}(\bm{\lambda}_t \sf{E}^\Sigma V,\ul{\xi}).\]
The claim follows immediately from this.
\end{proof}

In the special case $V=\g_\bC$ relevant to the equivariant Verlinde formula, one has
\[ \nabla \Tr_{\psi^j \g_\bC}=\sum_{\alpha \in \mf{R}}(2\pi \i)je^{j\alpha}\alpha=\sum_{\alpha \in \mf{R}_+}(2\pi \i)j(e^{j\alpha}-e^{-j\alpha})\alpha,\]
from which we find (taking $b=1$),
\[ v_t=\frac{1}{2\pi \i t}\sum_{\alpha \in \mf{R}_+}\sum_{j>0}\frac{(-t)^{j}}{j}(e^{j\alpha}-e^{-j\alpha})\alpha=\frac{1}{2\pi \i t}\sum_{\alpha \in \mf{R}_+}\log\Big(\frac{1+te^{-\alpha}}{1+te^\alpha}\Big)\alpha.\]

\bibliographystyle{amsplain}
\bibliography{../Biblio}

\end{document}